\documentclass[11pt, twoside, leqno]{article}

\usepackage{amssymb}
\usepackage{amsmath}
\usepackage{amsthm}
\usepackage{color}
\usepackage{mathrsfs}
\usepackage{txfonts}

\usepackage{appendix}

\usepackage{indentfirst}

\usepackage{enumerate}

\allowdisplaybreaks

\def\XXint#1#2#3{{\setbox0=\hbox{$#1{#2#3}{\int}$ }
\vcenter{\hbox{$#2#3$ }}\kern-.6\wd0}}

\def\red{\color{red}}

\newcommand{\E}{\mathbb{E}}
\newcommand{\R}{\mathbb{R}}

\newtheorem{theorem}{Theorem}[section]
\newtheorem{lemma}[theorem]{Lemma}
\newtheorem{corollary}[theorem]{Corollary}

\newtheorem{example}[theorem]{Example}
\newtheorem{proposition}[theorem]{Proposition}
\theoremstyle{definition}
\newtheorem{definition}[theorem]{Definition}
\newtheorem{remark}[theorem]{Remark}
\newcounter{rea}
\renewcommand{\appendix}{\par
   \setcounter{section}{0}%
   \setcounter{subsection}{0}%
   \setcounter{subsubsection}{0}%
   \gdef\thesection{\@Alph\c@section}%
   \gdef\thesubsection{\@Alph\c@section.\@arabic\c@subsection}%
   \gdef\theHsection{\@Alph\c@section.}%
   \gdef\theHsubsection{\@Alph\c@section.\@arabic\c@subsection}%
   \csname appendixmore\endcsname
 }

\numberwithin{equation}{section}

\begin{document}

\arraycolsep=1pt

\title{\bf\Large  Products and Commutators of
 Martingales in $H_1$ and  ${\rm BMO}$
	\footnotetext{\hspace{-0.35cm} 2020 {\it
			Mathematics Subject Classification}. Primary 60G42;
		Secondary 60G46, 47B47, 42B25, 42B30.
		\endgraf {\it Key words and phrases.} martingale Hardy space, martingale BMO,
		bilinear decomposition, commutator, dyadic Hilbert transform.
		\endgraf This project is supported by  the National
		Natural Science Foundation of China
		(Grant Nos. 12001541, 12125109, 12201647, 11971058 and 12071197), the National Key
		Research and Development Program of China
		(Grant No. 2020YFA0712900), and the Natural Science Foundation of Hunan Province
		(Grant Nos. 2021JJ40711 and 2021JJ40714).}}
\author{Aline Bonami\footnote{Corresponding author,
E-mail: \texttt{aline.bonami@univ-orleans.fr}/{\red February 06, 2023}
/Final version.},\ \ Yong Jiao, Guangheng Xie, Dachun Yang and Dejian Zhou}
\date{}
\maketitle

\vspace{-0.7cm}

\begin{center}
\begin{minipage}{13cm}
{\small {\bf Abstract}\quad
Let $f:=(f_n)_{n\in \mathbb{Z}_+}$ and $g:=(g_n)_{n\in \mathbb{Z}_+}$ be
two martingales related to the probability space
$(\Omega,\mathcal F,\mathbb P)$ equipped with the
filtration $(\mathcal F_n)_{n\in \mathbb{Z}_+}.$ Assume that
$f$ is in the martingale Hardy space $H_1$ and $g$ is
in its dual space, namely the martingale $\rm BMO.$ Then the semi-martingale
$f\cdot g:=(f_ng_n)_{n\in \mathbb{Z}_+}$ may be written as the sum
$$f\cdot g=G(f, g)+L( f,g).$$
Here $L( f,g):=(L( f,g)_n)_{n\in\mathbb{Z}_+}$
with $L( f,g)_n:=\sum_{k=0}^n(f_k-f_{k-1})(g_k-g_{k-1)})$
for any $n\in\mathbb{Z}_+$, where $f_{-1}:=0=:g_{-1}$.
The authors prove that $L( f,g)$ is a process with bounded
variation and limit in $L^1,$ while  $G(f,g)$ belongs to
the martingale Hardy-Orlicz space
$H_{\log}$ associated with the Orlicz function
$$\Phi(t):=\frac{t}{\log(e+t)},\quad \forall\, t\in[0,\infty).$$
The above bilinear decomposition $L^1+H_{\log}$ is sharp in the sense that,
for particular martingales, the  space $L^1+H_{\log}$
cannot be replaced  by a smaller space having a larger dual.
As an application, the authors  characterize the largest subspace of $H_1$,
denoted by $H^b_1$ with $b\in {\rm BMO}$, such that the commutators $[T, b]$ with
classical sublinear operators $T$ are bounded from $H^b_1$ to $L^1$.
This endpoint boundedness of commutators allow the authors to give more applications.
On the one hand, in the martingale setting, the authors
obtain the endpoint estimates of commutators for both
martingale transforms and  martingale fractional
integrals. On the other hand, in harmonic analysis,
the authors establish the endpoint estimates of commutators
both for the dyadic Hilbert transform beyond doubling measures
and for the maximal operator of Ces\`{a}ro means of
Walsh--Fourier series.
}
\end{minipage}
\end{center}

\vspace{0.1cm}

\section{Introduction}

Motivated by developments in
geometric function theory and nonlinear elasticity,
Bonami et al. in their pioneer work \cite{BIJZ2007}
investigated the linear decomposition of
products of two functions respectively in the Hardy space $H_1(\mathbb{R}^n)$ and
the space ${\rm BMO}(\mathbb{R}^n)$ of
functions of bounded mean oscillation and conjectured in \cite[Conjecture 1.7]{BIJZ2007}
that this linear decomposition should be bilinear.
Via the wavelet multiresolution analysis, Bonami et al. \cite{BGK2012}
completely solved this conjecture by proving that
there exist two bounded bilinear operators
$$L:\ H_1(\mathbb{R}^n)\times {\rm BMO}(\mathbb{R}^n)\to L^1(\mathbb{R}^n)$$
and
$$G:\ H_1(\mathbb{R}^n)\times {\rm BMO}(\mathbb{R}^n)\to H_{\varphi}(\mathbb{R}^n)$$
such that the product $fg$ of any $f\in H_1(\mathbb{R}^n)$
and $g\in {\rm BMO}(\mathbb{R}^n)$
can be represented as
$$fg=L(f,g)+G(f,g).$$
Here the product is taken in the sense of Schwartz distributions
and $H_{\varphi}(\mathbb{R}^n)$ is
a Musielak--Orlicz Hardy space related to the Musielak--Orlicz
function
\begin{align}\label{1.1x}
\varphi(x,t):=\frac {t}{\log(e+|x|)+\log(e+t)}
\end{align}
for any $(x,t)\in \mathbb{R}^n\times[0,\infty)$.
The question of the optimality of the result was raised in this article,
to know whether $H_{\varphi}(\mathbb{R}^n)$ can be replaced by
a smaller vector space. The pointwise multiplier theorem of
Nakai and Yabuta \cite{NY1985}
allowed the authors in \cite{BGK2017} to answer that the smallest
Banach space containing
$H_{\varphi}(\mathbb{R}^n)$ is in some sense the smallest Banach space
containing these products. Optimality was deduced in one dimension in
\cite{bk14} from an exact factorization and in \cite{BGK2017} for $n\geq
2$ from a weak factorization.
More related progress on this subject over $\mathbb R^n$ can be found
in \cite{cky18,blyy21,yyz21,zyy22}.
Other contexts than $\R^n$ have also been studied recently. In particular,
Fu et al. \cite{fyl17} and Liu et al. \cite{LYY2018}
established  bilinear decomposition on  metric
measure spaces of homogeneous type.

It is natural to consider the product of general martingales. It turns
out that the bilinear decomposition appears particularly intuitive in the
context of martingales, even if new difficulties appear. Let us first fix
some symbols before describing this new situation.

Let $(\Omega,\mathcal F,\mathbb P)$ be a complete probability space and
$(\mathcal F_{n})_{n\in{\mathbb Z_+}}$ an increasing sequence of
sub-$\sigma$-algebras of $\mathcal F$ satisfying
$\mathcal F=\sigma(\bigcup_{n\in\mathbb Z_+}\mathcal F_n)$.
We assume for simplicity that $\mathcal{F}_0=\{\emptyset, \Omega\}$.
The martingale Hardy space and the martingale ${\rm BMO}$
are denoted, respectively, by $H_1$ and ${\rm BMO}$.
Martingales $f=(f_n)_{n\in \mathbb{Z}_+}$ that belong to $H_1$ or
${\rm BMO}$ can be identified with their limits $f_\infty,$
so that these martingale spaces may be seen as spaces of
functions (or random variables) on $\Omega,$ but it is no more
the case for the products we are interested in:
if $f=(f_n)_{n\in \mathbb{Z}_+}$  belongs to $H_1$ and
$g=(g_n)_{n\in \mathbb{Z}_+}$ belongs to ${\rm BMO},$
the product $f_\infty g_\infty$ is not integrable
in general, so that such an identification is not
possible. Recall that it is possible in the sense
of distributions in the case of $\R^n.$ Here it is
natural to define $f\cdot g$ as the discrete process
$(f_n g_n)_{n\in Z_+}.$ But it is not a martingale.
Fortunately it is a semi-martingale, which is the sum
of a martingale and a process with bounded variation.
Our main result  gives the decomposition of the
semi-martingale $f\cdot g$ into two parts, one a martingale
in $H_{\log}$ and one a bounded variation process.
We call $\mathcal BV$ the space of such processes,
that is, the space of adapted sequences of random variables,
$h=(h_n)_{n\in Z_+}$, such that
$$\E\left(\sum_{n=1}^{\infty} \left|h_n-h_ {n-1}\right|\right)<\infty,$$
where $\E$ denotes the expectation.
On the other hand, $H_{\log}$ is the martingale
Orlicz Hardy space associated with the Orlicz function
\begin{align}\label{1.1y}
\Phi(t):=\frac{t}{\log(e+t)}, \ \forall\,t\in[0,\infty).
\end{align}

For an adapted process $h=(h_n)_{n\in\mathbb{Z}_+},$ we let
$d_0(h):= h_0$ and $d_n(h)=d_nh:=h_n-h_{n-1}$ for any $n\in\mathbb N$.
The main result of our first part is the following one.
\begin{theorem}\label{thm-bilinear}
For any $(f,g)\in H_1\times {\rm BMO},$ one can write
	$$f \cdot g= L(f,g)+G(f,g),$$
where $L$ and $G$ are two bounded bilinear operators with
	$$L:\ H_1\times {\rm BMO}\to \mathcal BV$$
	and
	$$G:\ H_1\times {\rm BMO}\to H_{\log}.$$
	Moreover, for any $n\in{\mathbb Z}_+,$
	$$L(f,g)_n:=\sum_{k=0}^n d_k(f)d_k(g).$$
\end{theorem}

The bilinearity of $L$ is clear by its definition and follows immediately for $G.$
The martingale $G(f, g)$ is the sum of two paraproducts, $\Pi_1(f,g)$
and $\Pi_2(f,g),$ which are respectively such that
$d_n(\Pi_1(f,g))=f_{n-1}d_n(g)$ and
$d_n(\Pi_2(f,g))=g_{n-1}d_n(f)$ for any $n\in{\mathbb Z}_+$.

\begin{remark}
\begin{enumerate}[{\rm (i)}]
\item We will show that Theorem \ref{thm-bilinear}
is somewhat sharp in the sense of duality. For this,
we will rely this theorem to its dual and use the
characterization of pointwise multipliers of $\mathrm{BMO}$
given in \cite{NS2014}.
Precisely, if  Theorem \ref{thm-bilinear} holds true for $\mathcal Y$ with
$\mathcal Y\subset H_{\log}$,
then
$$\left(L^1+\mathcal Y\right)^{\ast}
=\left(L^1+H_{\log}\right)^{\ast}.$$

\item
As quoted above, in the same problem on $\R^n$,
the Orlicz function $\Phi$ in \eqref{1.1y} that appears in the definition
of $H_{\log}$ is replaced by the   Musielak--Orlicz function
$\varphi$ in \eqref{1.1x}; see \cite{BGK2012}.
The dependence in $x$ is there for the behavior at $\infty$ in $\R^n$ and
does not appear in local results or periodic ones; see, for instance,
\cite{BIJZ2007,cky18,yyz21,zyy22} in the Euclidean space. It is natural to find the same
Orlicz function here as that for periodic functions on $\R$ or for the
dyadic situation which was studied by Bakas et al. \cite{BPRS2020}.

\item We would like to mention that Odysseas Bakas, 
Zhendong Xu, Yujia Zhai, and Hao Zhang \cite{BXZZ}
have independently obtained the analog of Theorem \ref{thm-bilinear}. They then developed very interesting generalizations to $H^p$ for any $p\in (0,1)$ and to martingales related to non probability measures, while our article takes another direction.
It is worth mentioning that in Theorem \ref{thm-bilinear} 
we insist on the meaning of the product. An analog of the definition 
of the product as a distribution as in the classical case does not 
seem available in general. This is why we define the product as 
a semi-martingale, namely a sum of a martingale and a process with bounded variation.
\end{enumerate}	
\end{remark}

The next observation will be central in our further developments and
gives another way to see the products involved in the theorem.
Observe that $L(f,g)$ converges in $L^1$ and almost surely (for short a.s.).
When $n\to\infty$, since $f_n g_n$ converges also a.s., the same is valid for
$G(f,g)_n$ but the fact that we only know that it is
in $H_{\log}$ does not allow us to recover the martingale
from its a.s. limit. Nevertheless it is  the case when
$f$ is an $L^2$-martingale and, in this case, the decomposition can
also be written in terms of limit values at $\infty.$ In general, the $H_{\log}$
part of the product can be recovered as a limit,
using the density of $L^2$-martingales in $H_1.$

\medskip

One remarkable application of the aforementioned
bilinear decomposition in harmonic analysis
is due to Ky \cite{K2013,K2015}.
It is well known that commutators
generated by both Calder\'on--Zygmund operators and ${\rm BMO}(\mathbb R^n)$ functions
may not map $H_1(\mathbb R^n)$ continuously into $L^1(\mathbb R^n)$.
Using the bilinear decomposition, Ky \cite{K2013}
characterized the largest subspace $\mathcal Z$ of $H_1(\mathbb R^n)$
such that most classical commutators
are bounded from this subspace $\mathcal Z$ to $L^1(\mathbb R^n)$. The second goal
of this article is to adapt this characterization in
the martingale setting and establish  endpoint
estimates of commutators
generated by both martingale operators and multiplications by
${\rm BMO}$ functions.

Commutators of martingale transforms
were first investigated by Janson \cite{J1981}
and then studied by Chao and Peng \cite{CP1996} for regular martingales.
Commutators of martingale transforms for non-regular martingales
were recently investigated by Treil \cite{Tr2013}.
In addition, the boundedness of commutators of martingale fractional integrals were
developed by Chao and Ombe \cite{CO1985} and very recently
by Nakai et al. \cite{NS2019,ANS2020}.
However, up to now, there does not exist any endpoint
estimate of martingale commutators.

More precisely, let $b\in \mathrm{BMO}$, $q\in[1,\infty),$ and
$T$ be in a class $\mathcal K_q$ of sublinear operators
containing almost all important operators in the
martingale setting (see Definition \ref{def-K} for its definition).
The {\it commutator $[T,b]$} of the sublinear operator $T$
is defined by setting (when it makes sense), for any
$f\in H_1$ and $x\in \Omega$,
$$[T,b](f)(x):=T\left(bf-b(x)f \right)(x).$$
Moreover, if $T$ is linear, then
$[T,b](f)=T(bf)-bT(f).$
Now, we establish the (sub)bilinear
decomposition for the commutator $[T,b]$ as follows.
\begin{theorem}\label{bdT}

Let  $q\in[1,\infty)$ and $T\in \mathcal{K}_q$.
Then there exists  a bounded subbilinear operator
$$R:\ H_1\times \mathrm{BMO} \to L^q$$
such that, for any  $(f,b)\in H_1\times \mathrm{BMO}$,
$$|T(L(f,b))|-R(f,b)\leq |[T,b](f)|\leq R(f,b)+|T(L(f,b))|,$$

  In particular, if $T$ is linear, then, for any
  $(f,b)\in H_1\times \mathrm{BMO}$, the bilinear operator $R(f,b)$ defined by setting
$$R(f,b):=[T,b](f)-T(L(f,b))$$
  is bounded from $\ H_1\times \mathrm{BMO}$ into $L^q.$
	\end{theorem}
This means that continuity properties for the
commutator may be deduced from continuity properties for $T(L(f, b)).$
For any $b\in \mathrm{BMO}$, we introduce a new martingale
Hardy space $H_1^b$ (see Definition \ref{def-h1b} below)
which is related to
endpoint estimates  of commutators $[T,b]$ for any $T\in
\mathcal{K}_q$.

\begin{theorem}\label{cmT}
	Let  $q\in[1,\infty)$, $T\in \mathcal{K}_q$, and
	$b\in \mathrm{BMO}.$
	Then there exists a positive constant $C$ such that,
	for any $f\in H_1^b,$
	$$\|[T,b](f)\|_{L^q}\leq C\|f\|_{H_1^b}.$$
\end{theorem}

\begin{remark}
Let $b\in \mathrm{BMO}$. It is worth   mentioning that the space
$H_1^b$ in Theorem \ref{cmT} is sharp in the sense that
$\mathcal{Y}:=H_1^b$ is the largest subspace of $H_1$
such that, for any $T\in \mathcal{K}_1$,
the commutator $[T,b]$ is bounded from $\mathcal{Y}$ to $L^1 $;
	see Remark \ref{rem-commutator} below.
\end{remark}
\begin{remark}
	By analogy with  the classical case on $\mathbb R^n$ (see \cite{Per, K2013}),
	examples of functions in $H_1^b$ are given by atoms $a$
	related to the integer $n\in\mathbb Z_{+}$ such that $\E_{n} (a b)=0 $
	as well as sums of such atoms.
\end{remark}
We point out that our results have wide applications in martingale
theory and harmonic analysis.
On the one hand, we obtain the endpoint boundedness of the commutators
of both martingale transforms and martingale fractional
integrals (see Section \ref{exampleK} below).
Note that these endpoint estimates were not considered before.
So, these estimates complete the story of martingale
commutators investigated by Janson \cite{J1981},
Chao et al. \cite{CO1985,CP1996}, Nakai et al.
\cite{NS2019,ANS2020}, and Treil \cite{Tr2013}.

On the other hand, we provide some applications in harmonic analysis.
In recent years, dyadic operators have attracted
a lot of attention related to the so-called
$A_2$-conjecture in harmonic analysis.
Especially, the boundedness of the dyadic Hilbert transform
beyond doubling measures (also known as the dyadic shift; see, for instance, \cite{Pe2019})
was first characterized
by L\'opez-S\'anchez et al. \cite{LLP2014}.
Motivated by this, we establish the endpoint estimate of
the commutator for the dyadic Hilbert transform beyond doubling measures.
Additionally, we establish the endpoint estimate of
the commutator for the maximal operator
of Ces\`{a}ro means of Walsh--Fourier series.
To the best of our knowledge, the commutator for the
maximal operator of Ces\`{a}ro means of Walsh--Fourier
series had not been investigated before.

Observe that the Vilenkin system is a natural generalization of the Walsh system
(see, for instance, \cite{PSTW22,PTW22}). It is interesting to see whether or not
our methods still work for the Vilenkin system.

The remainder of this article is organized as follows.

In Section \ref{sect-preliminary}, we present some notation and
preliminaries about the martingale Hardy space,
the martingale {\rm BMO} space, and the atomic
decomposition  that we use to prove our main results.

Section \ref{sect-bilinear} is devoted to proving
Theorem \ref{thm-bilinear}. Once written the product
of martingales via martingale paraproducts, write
$$f\cdot g=\Pi_1(f,g)+\Pi_2(f,g)+L(f,g),
\ \forall\,(f,g)\in H_1\times {\rm BMO} ,$$
the problem is reduced to the study of these three bilinear operators.
The main difficulty is to find a suitable decomposition of martingales
from the martingale Hardy space $H_1$.
Recall that wavelets or atomic decompositions  are used in
$\R^n$ (see \cite{BCKLYY2019,BGK2012}). However,
the martingale Hardy space $H_1$ associated with
the martingale square function does not admit
a classical atomic decomposition when the underlying filtration is not
regular. To overcome this difficulty, we use the Davis decomposition
(see Lemma \ref{lem-Davis} below) and
decompose the martingale $f\in H_1$ into two parts. The first one is
in the martingale Hardy space $ h_1$ defined by the
martingale conditional square operator and has an atomic decomposition.
The second one is in the
 $ h_1^d$, so that it is an $\ell^1$ sum of integrable jumps.
 So it is sufficient to estimate $\{\Pi_i(a,g)\}_{i=1}^2$,
 where $a$ is an  atom (or a  jump martingale, that is,
 a martingale $h$ such that $d_k h$ is $0$ except for one value of $k$).

In Section \ref{sect-commutator}, we establish the
(sub)bilinear decomposition of commutators and the
endpoint estimate of commutators.
We first introduce a class $\mathcal K_q$, with $q\in [1,\infty)$, of sublinear operators.
Comparing with \cite{K2013}, we do not limit ourselves only
to the case $q=1$ and benefit from this is to treat
martingale fractional integral operators.
Applying Theorem \ref{thm-bilinear},
we first establish a subbilinear decomposition of the
sublinear commutator $[T,b]$ (see Theorem \ref{bdT} below),
and then establish an equivalent characterization of
the martingale Hardy-type space $H_1^b$ via the bilinear
operator $L$ (see Theorem \ref{4ec} below). All these would
suffice for us to show Theorem \ref{cmT}.
In Section \ref{exampleK}, we provide  examples of operators in  the class $\mathcal K_q,$
such as martingale transforms and  martingale fractional integral operators.

Finally, Section \ref{sect-application} focuses on applications
of our results in harmonic analysis. In Section
\ref{subsect-hilbert}, we obtain the endpoint
estimate of the commutator for the dyadic
Hilbert transform beyond doubling measures.
Section \ref{CWF} contains the endpoint estimate
of commutators of the maximal operator of
Ces\`{a}ro means of Walsh--Fourier series.

Throughout this article, we always let $\mathbb{N}:=\{1,2, \ldots\}$,
$\mathbb{Z}_+:=\mathbb{N} \cup \{0\}$,  and
$\mathbb{Z}:=\{0,\pm 1,   \ldots \}$, respectively.
We use $C$ to denote a positive constant, which may differ from line to line.
The symbol $f\lesssim Cg$ means that there exists
a positive constant $C$ such that $f\leq Cg$. If
we write $f\approx g$, then it stands for $f\lesssim g$
and $g\lesssim f$. If $f\leq Cg$ and $g=h$ or $g\leq h$,
we then write $f\lesssim g\approx h$ or $f\lesssim g\lesssim h$,
rather than $f\lesssim g=h$ or $f\lesssim g\leq h$.
For any subset $E$ of $\Omega$, we use ${\mathbf{1}}_E$
to denote its \textit{characteristic function}.
For any measurable function $f$, define
$\mathrm{supp}\,(f):=\{x\in\Omega:\ f(x)\neq0\}$.

\section{Preliminaries}\label{sect-preliminary}
This section includes some basic
preliminary background concerning martingale Hardy spaces and BMO spaces
 that are needed throughout this article.
Our notation and terminology are standard as may be
found in  monographs \cite{Ga1973, Lo1993,We1994}.

\subsection{Martingale Hardy spaces}\label{subsect-hardy}

Let $(\Omega,\mathcal F,\mathbb P)$ be a complete probability space and
$(\mathcal F_{n})_{n\in{\mathbb Z_+}}$   an increasing sequence of
sub-$\sigma$-algebras of $\mathcal F$ satisfying
$\mathcal F=\sigma(\bigcup_{n\in\mathbb Z_+}\mathcal F_n)$.
The expectation operator with respect to $\mathcal F$
is denoted by $\mathbb E$.
The conditional expectation operators with respect to $(\mathcal F_n)_{n\in\mathbb Z_+}$
are denoted, respectively, by $(\mathbb E_n)_{n\in\mathbb Z_+}$. The sequence
$f:=(f_{n})_{n\in \mathbb{Z}_+}\subset L^1$ is called a {\it martingale} if,
 for any $n\in \mathbb{Z}_+$,
$$\E_n\left(f_{n+1}\right)=f_n.$$
Denote by $\mathcal M$ the set of all the martingales $f:=(f_n)_{n\in\mathbb Z_+}$
related to $(\mathcal F_n)_{n\in\mathbb Z_+}$.
For any $f\in\mathcal M$, define its {\it martingale difference}
by setting (with convenience, $f_{-1}:=0$ and $\mathbb{E}_{-1}:=0$)
\[d_n(f)=d_nf:=f_n-f_{n-1},\quad\forall\, n\in\mathbb{Z}_+.\]
As usual, for any martingale $f\in \mathcal{M}$ and any
$p\in[1, \infty]$, let
$$\|f\|_{L^p}:=\sup_{n\in \mathbb{Z}_+}\|f_n\|_{L^p}
:=\sup_{n\in \mathbb{Z}_+}\left(\int_{\Omega}
\left|f_n\right|^p\,d\mathbb{P}\right)^{\frac1p}.$$
If $\|f\|_{L^p}<\infty$, then $f$ is called an {\it $L^p$-bounded martingale}.

\begin{remark}\label{ffinf}
If $p\in(1,\infty)$ and $f\in \mathcal{M}$ is an
$L^p$-bounded martingale, then there exists an
$f_{\infty}\in L^p$ such that $f_n=\E_n(f_{\infty})$
for each $n\in \mathbb{Z}_+$ and $\|f_n-f_{\infty}\|_{L^p}\to 0$
as $n\to \infty$; see, for instance, \cite[p.\,28]{Lo1993}.
In this case, we will also speak of the {\em $L^p$-martingale.}
Recall that this is no more true for an $L^1$-bounded martingale $f:=(f_n)_{n\in\mathbb{Z}_+}$.
If a martingale $(f_n)_{n\in \mathbb{Z}_+}$ is
generated by a measurable function $f_{\infty}\in L^1$, that is,
$f_{n}=\mathbb{E}_n(f_{\infty})$ for each
$n\in\mathbb{Z}_+$, then, in this case, one also has
$\|f_n-f_{\infty}\|_{L^1}\to0$ as $n\to\infty$
and we will also speak of the {\em $L^1$-martingale.}

When a martingale $f:=(f_n)_{n\in \mathbb{Z}_+}$ is generated by a
function $f_{\infty}$, that is, when $f$ is an
{\em $L^1$-martingale,}  we will not distinguish
at times the martingale
and $f_{\infty}$.
\end{remark}

For any $n\in\mathbb \mathbb{Z}_+$, the {\it Doob maximal operators}
 $M(f)$, the {\it square function}  $S(f)$,
and the {\it conditional square functions}  $s(f)$ of a
martingale $f$ are defined, respectively, by setting
\[
 M(f):=\sup_{n\in\mathbb Z_+}{|\mathbb E_n(f)|},
\
S(f):={\left(\sum_{n\in \mathbb{Z}_+}{|{d_nf}|}^2\right)}^\frac12,\]
and \[
s(f):=\left(\sum_{n=1}^\infty
\mathbb E_{n-1}|d_nf|^2+|d_0f|^2\right)^\frac12.\]

\begin{definition}\label{def-hardy}
The {\it martingale Hardy spaces} $H_1$, $h_1$,
and $h_1^d$ are defined, respectively, by setting
$$H_1:=\left\{f\in\mathcal M:\ \|f\|_{H_1}
:=\|S(f)\|_{L^1}<\infty \right\},$$

$$h_1:=\left\{f\in\mathcal M:\ f_0=0,\ \|f\|_{h_1}
:=\|s(f)\|_{L^1}<\infty \right\},$$
and
$$h_1^d:=\left\{f\in\mathcal M:\ \|f\|_{h_1^d}:=
\sum_{k\in \mathbb{Z}_+}\|d_kf\|_{L^1}<\infty\right\}.$$
\end{definition}
Clearly, martingales in $H_1$ are $L^1$-martingales.

The following is the famous Davis inequality;
see \cite{Da1970} or \cite[Theorem 2.1.9]{Lo1993}.

\begin{lemma}\label{davis-ine}
There exists a  constant $C\in [1,\infty)$ such that, for any  $f\in H_1$,
$$
	C^{-1}\|f\|_{H_1}\leq \|M(f)\|_{L^1}\leq C\|f\|_{H_1}.
$$
\end{lemma}

The following result is a part of  \cite[Lemma 2.15]{We1994}.
\begin{lemma}\label{lem-Davis}
For any $f\in H_1$,  there exist a positive constant $C$
and two martingales $f^1\in h_1$ and $f^d\in h_1^d$ such
that $f=f^1+f^d$ with
	\begin{align}\label{davis-dec}
		\left\|f^1\right\|_{h_1}\leq C\|f\|_{H_1}\quad\mbox{and}\quad
		\left\|f^d\right\|_{h_1^d}\leq C\|f\|_{H_1}.
	\end{align}
\end{lemma}
 We will need to use at the same time the fact that $f\in H_1$
 and $f$ belongs to $L^2.$ Going back to the proof,
 one has the following property.
 \begin{remark}\label{sup-davis}
In Lemma \ref{lem-Davis}, if, moreover, $f$ is an $ L^2$-martingale,
 	then both $f^1$ and $f^d$ may be chosen so that
 	they are also  $ L^2$-martingales.
\end{remark}

The {\it Orlicz space} $L^{\log}$ is defined to
be the set of all the measurable functions $f$ such that
\[\|f\|_{L^{\log}}
:=\inf\left\{\lambda\in (0,\infty):\ \int_{\Omega}\frac{|f|/\lambda}{\log(e+|f|/\lambda)}
\,d\mathbb{P}\leq 1\right\}<\infty.\]
Clearly, $\|\cdot\|_{L^{\log}}$ is a quasi-norm and $L^1\subset L^{\log}$ with continuous embedding.
If we replace   $\|\cdot\|_{L^1}$ in Definition \ref{def-hardy}
therein by $\|\cdot\|_{L^{\log}}$, then we obtain the
definition of the {\it martingale Hardy spaces $H_{\log}$ and $h_{\log}$}.

The following result is a part of both
\cite[Theorem 2.11]{We1994} and \cite[Theorem 2.5]{MNS2012}.
\begin{lemma}\label{mar-ine}
\begin{enumerate}
\item[{\rm (i)}]
There exists a positive constant $C$ such that, for any $f\in h_1$,
	$$\|f\|_{H_1}\leq C \|f\|_{h_1}.$$
	
\item[{\rm (ii)}]
There exists a positive constant $C$ such that, for any $f\in h_{\log}$,
	$$\|f\|_{H_{\log}}\leq C \|f\|_{h_{\log}}.$$
\end{enumerate}
\end{lemma}

\subsection{Martingale BMO spaces and John--Nirenberg inequality}\label{subsec-bmo}

This section is devoted to definitions and basic results concerning  martingale BMO spaces.
Let $X_1$ and $X_2$ be two quasi-normed spaces both of which are subspaces
of some Hausdorff topological vector space.
The {\it space} $X_1+X_2$ is defined to be the set of all the elements $x$ of the
form $x=x_1+x_2,$ where $x_1\in X_1$ and $x_2\in X_2$, and is equipped
with the quasi-norm
$$\|x\|_{X_1+X_2}:=\inf\left\{\left\|x_1\right\|_{X_1}+\left\|x_2\right\|_{X_2}\right\},$$
where the infimum is taken over all
the elements $x_1\in X_1$ and $x_2\in X_2$ whose sum is equal to $x.$
The {\it space} $X_1\cap X_2$ is defined to be the set of all
the elements $x\in X_1\cap X_2$ and is equipped
with the quasi-norm
$$
	\|x\|_{X_1\cap X_2}
	:=\max\left\{\left\|x\right\|_{X_1},\left\|x\right\|_{X_2}\right\}.
$$

\begin{definition}\label{def-bmo}
	For any $p\in[1,\infty)$, the
	{\it martingale space ${\rm BMO}_p$ }
	is defined to be the set of
	all the martingales $f\in L^p$ with the norm
	$$\|f\|_{{\rm BMO}_p}:=\sup_{n\in\mathbb Z_{+}}\left\|\mathbb{E}_n
	\left(\left|f-f_{n-1}\right|^p\right)\right\|^{\frac1p}_{L^{\infty}}<\infty.$$
	The
	{\it martingale space ${\rm bmo}_p$ }
	is defined to be the set of
	all the martingales $f\in L^p$ with the norm
	$$\|f\|_{{\rm bmo}_p}:=\sup_{n\in\mathbb Z_{+}}\left\|\mathbb{E}_n
	\left(\left|f-f_{n}\right|^p\right)\right\|^{\frac1p}_{L^{\infty}}<\infty.$$
		The
	{\it martingale space ${\rm bmo}^d$ }
	is defined to be the set of
	all the martingales $f\in L^{\infty}$ with the norm
	$$\|f\|_{{\rm bmo}^d}:=\sup_{n\in\mathbb Z_{+}}\left\|d_nf\right\|_{L^{\infty}}<\infty.$$
\end{definition}

It is well known that, for any $p\in[1,\infty)$,
$$L^{\infty}\subset {\rm BMO}_p\subset {\rm bmo}_p\subset L^p;$$
see Weisz \cite[p.\,51]{We1994} for more details.
In the case $p=2$, it is obvious that
$$
	{\rm BMO}_2={\rm bmo}_2\cap {\rm bmo}^d
$$
with equal norms.
The John--Nirenberg inequality states that, for any given $p\in [1,\infty)$,
${\rm BMO}_p={\rm BMO}_1$
with equivalent norms.
We refer the reader to  \cite{Ga1973} or
\cite[Chapter 4]{Lo1993} for more details and also Nakai and Sadasue \cite{NS2017}
for more related studies.
Henceforth, in the sequel, we simply write ${\rm BMO}$
instead of ${\rm BMO}_p$ for any given $p\in[1,\infty)$.

Now,  we recall the dual theorem on martingale Hardy spaces.
For any (quasi) Banach space $X$, we denote by $X^{\ast}$ the {\it dual space} of $X$,
namely, the space of all continuous linear functional on $X$.
\begin{lemma}\label{duality}
The following duality results hold true:
	\begin{enumerate}[{\rm(i)}]
		\item $(H_1)^{\ast}={\rm BMO}$ with equivalent norms;
		\item $(h_1)^{\ast}={\rm bmo}_2$ and $(h_1^d)^*=\mathrm{bmo}^d$ with equivalent norms.
	\end{enumerate}
\end{lemma}

We refer the reader to \cite[Theorem 2.2.2]{Lo1993} for the proof of item (i)
and to both \cite[Theorem 2.23]{We1994} and \cite[Theorem 2.32]{We1994} for the proof of item (ii).

\subsection{Atomic decompositions of martingale Hardy spaces}

In this subsection, we introduce atomic decompositions of
Hardy spaces defined in Section \ref{subsect-hardy}.
We first recall the definition of atoms; see
the monograph \cite[Chapter 2]{We1994} of Weisz for more details.

\begin{definition}\label{def-simple-atom}
	A measurable function $a$ is called a
	{\it simple $(s,\infty)$-atom} if there exist an
	integer $n\in\mathbb{Z}_{+}$ and a set $A\in \mathcal F_n$ such that
	\begin{enumerate}[{\rm (i)}]
		\item $a_n:=\mathbb{E}_n(a)=0$;
		\item $\mathrm{supp}\,(a)\subset A$;
		\item $
		\|s(a)\|_{L^{\infty}}\leq \left[\mathbb{P}(A)\right]^{-1}.$
	\end{enumerate}
If the above {\rm (iii)} is replaced by $	\|a\|_{L^{\infty}}
\leq \left[\mathbb{P}(A)\right]^{-1}$, then one obtains
the definition of a {\it simple $\infty$-atom}.
\end{definition}

In the following lemma, we collect several  useful  properties related to simple $(s,\infty)$-atoms.
\begin{lemma}\label{TaC}
Let $a$ be a simple $(s,\infty)$-atom with respect to some $n\in\mathbb{Z}_+$ and $A\in \mathcal{F}_n$.
Then the following hold true:
\begin{enumerate}[{\rm (i)}]
	\item for any $T\in \{M, S, s\}$,  the support of $T(a)$ is contained in the set $A$;
	\item for any $p\in[1,2]$, one has $\|M(a)\|_{L^p}\leq 2 [\mathbb{P}(A)]^{1/p-1}$;
	\item for any  $p\in[1,2]$ and $T\in \{S,s\}$,
	one has $\|T(a)\|_{L^p}\leq  [\mathbb{P}(A)]^{1/p-1}$;
	\item  for any $p\in[1,2]$, one has $\|a\|_{L^p} \leq [\mathbb{P}(A)]^{1/p-1}$.
\end{enumerate}
In particular, $a$ is in $h_1.$
\end{lemma}
\begin{proof} We give the proof for completeness.
Note that $A\in\mathcal{F}_n$ and, for any $m\leq n$, $\mathbb{E}_m(a)=0$.
Thus, for any $m\in\mathbb{Z}_{+}$,
$\mathbb{E}_m(a)\mathbf{1}_A=\mathbb{E}_m(a)$ and $d_m(a)\mathbf{1}_A=d_m(a)$,
which yield item (i).
By item (i), for each  $p\in[1,2]$, we have
	\begin{align*}
		\|M(a)\|_{L^p} &=\|M(a)\mathbf{1}_A\|_{L^p}\leq \|M(a)\|_{L^2}
		\left[\mathbb{P}(A)\right]^{1/p-1/2}\\
		&\leq  2 \|a\|_{L^2} \left[\mathbb{P}(A)\right]^{1/p-1/2}=2 \|s(a)\|_{L^2}
		\left[\mathbb{P}(A)\right]^{1/p-1/2}\\
		&\leq  2  \|s(a)\|_{L^{\infty}} \left[\mathbb{P}(A)\right]^{1/2}
		\left[\mathbb{P}(A)\right]^{1/p-1/2}
		\leq 2\left[\mathbb{P}(A)\right]^{1/p-1},
	\end{align*}
where in the second inequality we used the Doob maximal inequality (see e.g. \cite[Theorem 2.1.3]{Lo1993}).
Item (iii) and item (iv) can be proved by the same argument as that
used in the proof of item (ii). Instead of the Doob maximal inequality,
we use the following basic facts
$$\|a\|_{L^2}=\|S(a)\|_{L^2}
=\|s(a)\|_{L^2}.$$
This finishes the proof of Lemma \ref{TaC}.
\end{proof}

Let us now turn to the  atomic decomposition, which is attributed
to Herz and found in   \cite{He1974}.
In this article, we use the following simple atomic decomposition of
$h_1$, which can be found in \cite[Theorem 2.5]{We1994} (see also \cite{W1990} and \cite[Chapter 5.3]{PSTW22}).

\begin{lemma}\label{lem-atomic-decom}
	Let $f\in h_1$. Then there exist a sequence $(a^k)_{k\in\mathbb{Z}}$
	of simple $(s,\infty)$-atoms, a sequence $(\mu_{k})_{k\in\mathbb{Z}}$
	of real numbers, and a positive constant $C$ such that, for any $n\in\mathbb{Z}_{+}$,
	$$f_n=\sum_{k\in\mathbb Z}\mu_{k}\mathbb{E}_n\left(a^k\right)
\quad \mbox{ a.s.}$$
and
$$\sum_{k\in \mathbb Z}\left|\mu_k\right|\leq C \|f\|_{h_1}.$$
\end{lemma}

 We will need to use at the same time the fact that $f\in h_1$ and $f$ is an $L^2$-martingale.
 Going back to the proof, one has the following property.
 \begin{remark}\label{sup-atomic}
 	In Lemma \ref{lem-atomic-decom}, if, moreover,
 	$f$ is an $L^2$-martingale, then the series
 	$\sum_{k\in\mathbb Z}\mu_{k}\mathbb{E}_n(a^k)$ converges also in $L^2.$
\end{remark}

\section{Products and Paraproducts}\label{sect-bilinear}
In this section, we prove  Theorem \ref{thm-bilinear}.
Take $(f,g)\in H_1\times {\rm BMO}.$ We may assume without loss of generality that $g_0=0$
(if not, there is a supplementary term in $H_1,$ which is contained in $H_{\log}$).
We start from the identity: for any $n\in\mathbb{N}$,
\begin{align}\label{decom}
f_ng_n&= 	\sum_{k=1}^n\sum_{j=1}^n (d_jf)( d_kg)\\
	&=\sum_{k=1}^nf_{k-1}(d_kg)+
	\sum_{k=1}^n (d_kf)g_{k-1}+
	\sum_{k=1}^n (d_kf)(d_kg)\nonumber\\
	&=:\Pi_1(f,g)_n+\Pi_2(f,g)_n+L(f,g)_n.\nonumber
\end{align}
Note first that, since $\sup_{n\in\mathbb Z_{+}} \|d_n g\|_{L^\infty}<\infty,$
it follows that each $g_n$ for any $n\in\mathbb Z_{+}$ is bounded.
So each term of these above expressions is integrable.
We recognize in the last  term the process of bounded
variation (we will prove this). Note that $\Pi_1$ and
$\Pi_2$ are paraproducts. They are martingales because, for every $k\in\mathbb N,$
    $$\mathbb{E}_{k-1}(f_{k-1}d_kg)=0\mbox{ and }\mathbb{E}_{k-1}(d_kfg_{k-1})=0.$$
 This induces that $(f_ng_n)_{n\in\mathbb Z_{+}}  $ is a semi-martingale.
    We  first prove that $L(f, g)$ has bounded variation,
    then that the martingale $\Pi_1(f, g)$ belongs to $H_1,$
    with the continuity of the mapping $(f,g)\mapsto \Pi_1(f, g)$,
    and the same for $\Pi_2,$ with $H_{\log}$ in place of $H_1.$
    Since $H_1$ is contained in $H_{\log},$ this will allow us to
    conclude the proof of Theorem \ref{thm-bilinear}.

\subsection{The process of bounded variation}
This is given by the following proposition.
\begin{proposition}\label{pro-3}
    Let $f:=(f_n)_{n\in \mathbb{Z}_+}\in H_1$ and $g:=(g_n)_{n\in \mathbb{Z}_+}\in \mathrm{BMO}$. Then
    $$\sum_{k\in \mathbb{Z}_+} |d_kf|\,|d_kg| \in L^1.$$ In particular the
    process $L(f,g)$ converges to a function in $L^1.$
\end{proposition}
\begin{proof}
From the proof of the duality $(H_1)^*=\mathrm{BMO}$
(see, for instance, \cite[p.\,43]{Lo1993}), one can deduce that
$$\sum_{n\in \mathbb{Z}_+}\E (|d_kfd_kg|)\leq \sqrt{2} \|f\|_{H_1} \|g\|_{\mathrm{BMO}}<\infty.$$
The desired result follows from the fact that $L^1$ is a complete space.
This finishes the proof of Proposition \ref{pro-3}.
\end{proof}

\subsection{ The paraproduct  $\Pi_1$}\label{sec-3Pi1}

This subsection aims to show the boundedness of
the bilinear operator $\Pi_1$ defined in \eqref{decom}.
In fact we prove more. That is, we have the following proposition.
\begin{proposition}\label{pro-1}
	The bilinear operator $\Pi_1$  is bounded from the product
	space $ H_1\times {\rm bmo_2}$ into the space $h_1$.
\end{proposition}

To show Proposition \ref{pro-1}, by the Davis decomposition in Lemma \ref{lem-Davis},
we can write $f$ as a sum of two martingales, one in $h_1$ and
the other one in $h_1^d.$ This leads us to consider separately
the cases of martingales in $h_1$ and $h_1^d.$
\begin{lemma}\label{Pi1-lem2}
    Let $f\in h_1^d$ and $g\in \mathrm{bmo}_2.$ Then the paraproduct $\Pi_1(f,g)$ is in $h_1,$ with
$$\left\|\Pi_1(f,g)\right\|_{h_1}\leq \|g\|_{{\rm bmo}_2} \left\|f\right\|_{h_1^d}.$$
\end{lemma}
\begin{proof}
Since $h_1$ is a Banach space, it is sufficient to consider separately each term of the product,
that is, each $\Pi_1(d_n(f), g)=d_n(f)\sum_{k\geq n+1} d_k(g).$ But, for any $n\in\mathbb Z_{+}$,
$$s\left(\Pi_1(d_n(f), g)\right)
=|d_n(f)|\left[\sum_{k=n+1}^{\infty} E_{k-1}\left(
\left[d_k(g)\right]^2\right)\right]^{1/2}.$$
We use the Schwarz inequality for the conditional
expectation related to $\mathcal F_n$ and write that, for any $n\in\mathbb Z_{+}$,
\begin{align*}
	\E_n \left(s\left(\Pi_1(d_n(f), g)\right)\right)
	&\leq |d_n(f)|\left[\E_n \left(\sum_{k=n+1}^{\infty}
	\E_{k-1}\left(\left[d_k(g)\right]^2\right)\right)\right]^{1/2}\\
	&=|d_n(f)|\left[\E_n\left( \sum_{k=n+1}^{\infty}
	\left[d_k(g)\right]^2\right)\right]^{1/2}.
\end{align*}
The second factor is bounded by the ${\rm bmo}_2$ norm of $g$, so that, for any $n\in\mathbb Z_{+}$,
$$\E_n \left(s\left(\Pi_1(d_n(f), g)\right)\right)
\leq |d_n(f)|\,\|g\|_{{\rm bmo}_2}.$$ We conclude by taking the
expectation of both sides, then the sum in $n\in\mathbb Z_{+}$.
This finishes the proof of Lemma \ref{Pi1-lem2}.
\end{proof}

In order to prove the proposition for martingales in  $h_1,$ we prove
 the following lemma which also gives an explicit constant, but for atoms.

\begin{lemma}\label{Pi1-lem1}
For any simple $(s,\infty)$-atom $a$ and any $g\in \mathrm{bmo}_2$,
	$$\|\Pi_1(a,g)\|_{h_1}\leq 2 \|g\|_{\mathrm{bmo}_2}.$$
\end{lemma}
\begin{proof}
Let $g\in \mathrm{bmo}_2$. Since $a$ is a simple $(s,\infty)$-atom,
it follows that there exist an $n\in\mathbb{Z}_{+}$ and
an $A\in \mathcal{F}_n$ such that   $\mathbb{E}_n(a)=0$ and $\mathrm{supp}\,(a)\subset A.$
So $a_k=\mathbb{E}_k(a)=0$ for any $k\in\{0,\ldots,n\}.$ Then
	\begin{align*}
		\Pi_1(a,g)&=\sum_{k\in \mathbb{Z}_+} a_{k-1} d_kg
		=\sum_{k= n+1}^{\infty} a_{k-1}  d_kg,
	\end{align*}
which, combined with the definitions of both $s$ and $\mathrm{bmo}_2,$
further implies that
	\begin{align*}
		s\left(\Pi_1(a,g)\right)
		&= \left[ \sum_{k= n+1}^{\infty}\mathbb{E}_{k-1}
		\left(|a_{k-1}|^2 |d_kg|^2\right)\right]^{1/2} \\
		&=\left[ \sum_{k= n+1}^{\infty}|a_{k-1}|^2
		\mathbb{E}_{k-1}\left(|d_kg|^2\right)\right]^{1/2}
		\leq  M(a)\|g\|_{\mathrm{bmo}_2}.
	\end{align*}
From the above argument and both (i) and (ii) of Lemma \ref{TaC}, we infer that
	$$ \left\|\Pi_1(a,g)\right\|_{h_1}\leq \|M(a)\|_{L^1}  \|g\|_{\mathrm{bmo}_2}
	\leq \|M(a)\|_{L^2} \mathbb P(A)^{1/2} \|g\|_{\mathrm{bmo}_2} \leq 2\|g\|_{\mathrm{bmo}_2},$$
	which completes the proof of Lemma \ref{Pi1-lem1}.
    \end{proof}

We are now ready to prove Proposition \ref{pro-1}.
\begin{proof}[Proof of Proposition \ref{pro-1}]
Take $(f,g)\in H_1\times {\rm BMO}$.
According to Lemma \ref{lem-Davis},
 there exists a decomposition $f=f^1+f^d$ such that $f^1\in h_1$
 and $f^d\in h_1^d$ with
	$$\left\|f^1\right\|_{h_1} \lesssim  \|f\|_{H_1}
	\mbox{ and } \left\|f^d\right\|_{h_1^d} \lesssim \|f\|_{H_1}.$$
By the atomic decomposition of  $h_1$ (Lemma \ref{lem-atomic-decom}), we have
	$$f^1=\sum_{k\in \mathbb{Z}}\mu_k a^k\ \mbox{a.s.}
	\mbox{ and }\sum_{k\in \mathbb{Z}}\left|\mu_k\right|
	\lesssim \left\|f^1\right\|_{h_1},$$
		where $(a^k)_{k\in \mathbb{Z}}$  is a sequence  of simple $(s,\infty)$-atoms.
	 Combining this and Lemmas \ref{Pi1-lem2} and \ref{Pi1-lem1}, we find that
	\begin{align*}
	\left	\|\Pi_1(f,g)\right\|_{h_1}
	&\leq \left\|\Pi_1\left(f^1,g\right)\right\|_{h_1}
	+\left\|\Pi_1\left(f^d,g\right)\right\|_{h_1}\\
		&\leq \sum_{k\in \mathbb{Z}}\left|\mu_k\right|\cdot
		\left\|\Pi_1\left(a^k,g\right)\right\|_{h_1}+
		\left\|\Pi_1\left(f^d,g\right)\right\|_{h_1}\\
		&\lesssim \left\|f^1\right\|_{h_1}\|g\|_{\mathrm{bmo_2}}
		+\left\|f^d\right\|_{h_1^d}\|g\|_{\mathrm{bmo_2}}
		\lesssim \|f\|_{H_1} \|g\|_{\mathrm{bmo_2}}.
\end{align*}
	This finishes the proof of Proposition \ref{pro-1}.
\end{proof}

\begin{remark}
We point out here that Proposition \ref{pro-1} can be
deduced from \cite[Corollary 6]{CL1992}
which was proved via stopping time argument.
Here, we provide a different proof of  Proposition \ref{pro-1}
by using the atomic decomposition of martingale Hardy spaces.
Moreover, our argument here leads us to establish the endpoint estimate of commutators
(see Section \ref{sect-commutator} below) via using the atomic decomposition.
\end{remark}

\subsection{Boundedness of the operator $\Pi_2$}

The goal of this subsection is to show that the bilinear operator $\Pi_2$
defined in \eqref{decom} is bounded from $H_1\times {\rm BMO}$
into $H_{\log}$, namely the following result.
\begin{proposition}\label{pro-2}
	The bilinear operator $\Pi_2$ is a bounded operator from the product
	space $H_1\times {\rm BMO}$ into the space $H_{\log}$.
\end{proposition}
\begin{proof}
The proof is straightforward. We can assume as before that $g_0=0.$
We do not need to consider separately $h_1$ and $h^d$ here. We write that
\begin{equation} \label{MS}
   S\left(\Pi_2\left(f, g\right)\right)
=\left[\sum_{n\in\mathbb N} |g_{n-1}|^2 d_n(f)^2
\right]^{\frac 12}\leq M(g) S(f).
\end{equation}
The present proposition follows then from the next two lemmas
which are well known.
The first one, which is an adaptation of John--Nirenberg's inequality
and may be found in \cite[p.\,131]{Lo1993}, says that $M(g)$ belongs
to the exponential class. Let us recall that a function $\psi$
belongs to the {\it exponential class} $\exp L$ if there exists
a positive constant $\alpha$ such that $\exp(\alpha \psi) $ is integrable.
The exponential class $\exp L$ is a Banach space and we can
take as the {\it Luxemburg norm} the quantity
$$\|\psi\|_{\exp L}
:=\inf\left\{\lambda\in(0,\infty):\ \E\left(\exp(\psi/\lambda)\right)\leq 2\right\}.$$
\begin{lemma}\label{JN}
	Let $g$ be a martingale in ${\rm BMO},$ with $g_0=0.$ Then $Mg $ is in
	the exponential class and there exists a positive constant $C$, independent of $g$, such that
$$\|Mg\|_{\exp L}\leq C \|g\|_{\rm BMO}.$$
\end{lemma}
We then conclude the proof of Proposition \ref{pro-2}
by using the following generalized H\"older inequality
which may be found for instance in \cite[Lemma 3.2]{VT} (see also \cite{BPRS2020}).
We give a short proof for completeness. We only write this lemma for both
functions on $\Omega$ and the Orlicz functions that we have in mind,
but it is valid in a general context.
\begin{lemma}\label{key-lem}
Let $\phi$ be a function in $ L^1$, and $\psi$ a function in the exponential class.
Then the  product $\phi \psi$ belongs to $L^{\log}$.
Moreover, there exists a positive constant $C$ such that
$$
	\|\phi \psi \|_{L^{\log}}\leq C\|\phi\|_{L^1}\|\psi\|_{\exp L}.
$$
    The positive constant $C$ does not depend on $\phi, \psi.$
\end{lemma}
\begin{proof} By homogeneity in the two factors of the product,
	we can assume that both norms are $1.$ Moreover, as in the
	proof of H\"older's inequality, this inequality is obtained
	from an elementary inequality. Here, we claim that,
	for any $s,t\in(0,\infty),$
$$\frac{st}{\log(e+st)}\leq t+e^s.$$
Indeed, this is certainly true for $s$ or $t$ less than $1.$
Assume that $1<s<\log(e+t).$ Then the left hand side is bounded by $t.$
Finally, if $e^s>t+e,$ the left hand side is bounded by $e^s.$
To conclude the proof of the present lemma, we replace $s$ and $t$,
respectively, by $\phi (x)$ and $\psi(x)$, and then integrate in $x$.
This finishes the proof of Lemma \ref{key-lem}.
\end{proof}

Now the conclusion of Proposition \ref{pro-2} is direct:
just replace $\psi$ by $M(g),$ which is in the exponential
class by Lemma \ref{JN}, and $\phi$ by $S(f).$
This then finishes the proof of Proposition \ref{pro-2}.
\end{proof}

\begin{remark}
We realize that the proof of Proposition \ref{pro-2} is identical to
\cite[Section 3.4]{BXZZ} even if the bibliography to which we refer is different.
In fact both rely on a strong property of {\rm BMO} functions.
The authors of  \cite {BXZZ} cited Garcia's book \cite{Ga1973} for
the fact that $M(g)$ for any $g\in$ BMO is still in {\rm BMO}, while we use directly
Long's book \cite{Lo1993}, in which John--Nirenberg's
inequality was directly given for $M(g)$ with $g\in$ BMO.
\end{remark}

\subsection{Proof of Theorem \ref{thm-bilinear}}\label{sec-pf-bd}

We  can now prove Theorem \ref{thm-bilinear} with the help of
Propositions  \ref{pro-3}, \ref{pro-1}, and \ref{pro-2}.
\begin{proof}[Proof of Theorem \ref{thm-bilinear}]
	Let $(f,g)\in H_1\times {\rm BMO}.$ Clearly, for any $n\in\mathbb Z_{+},$
	$$f_ng_n=\Pi_1(f,g)_n+\Pi_2(f,g)_n+L(f,g)_n.$$
	Then, according to Proposition \ref{pro-3}, $L$ is bounded from the product
	space $H_1\times {\rm BMO}$ into the space $L^{1}$.
		Let $G:=\Pi_1+\Pi_2$.
	Observe that $H_1\subset H_{\log}$ and $h_1\subset H_1$.
	A combination of  this observation with Propositions \ref{pro-1} and
	\ref{pro-2} yields that $G$ is bounded from $H_1\times {\rm BMO}$ into   $H_{\log}$.
	This finishes the proof of Theorem \ref{thm-bilinear}.
\end{proof}

Recall that the space $H_1$ can be as well defined as the space of
martingales $f$ such that $Mf$ belongs to $L^1,$ that is,
we can replace the martingale square operator $S$ by the
Doob maximal operator $M$. This is no more the case for
the space  $H_{\log}$  and we define
 $H^M_{\log}$ as the space of martingales $f$ for which $Mf$ is in $L^{\log}.$
 We have nevertheless the following statement.
 \begin{theorem}\label{thm-Doob}
Theorem \ref{thm-bilinear} holds true when the space
  $H_{\rm log}$ is replaced by
the space  $H^M_{\log}$.
\end{theorem}
\begin{proof}
Only Proposition \ref{pro-2} deserves to be modified. We
observe that the previous proof extends directly to $f\in h_1^d$ because
$$M\left(\Pi_2\left(f,g\right)\right)
\leq M(g)\sum_{n\in\mathbb Z_{+}} \left|d_n(f)\right|.$$
As for  $f\in h_1,$ we notice that
$$s\left(\Pi_2\left(f,g\right)\right)\leq M(g)\,s(f),$$
so that $\Pi_2(f,g)$ is in $h_{\log}$.
But this last space is a subspace of $H^M_{\log}$ (see,
for instance, \cite[(2.5)]{MNS2012}). This allows us to conclude also for
any $f\in h_1.$ To finish the proof of the present theorem, we use
the Davis decomposition in Lemma \ref{lem-Davis} for $f\in H_1$, write $f=f^1+f^d$ and
use the fact that $H_{\log}^M$ is a quasi-Banach space, so that
    $$\left\|\Pi_2(f, g)\right\|_{H_{\log}^M}
    \lesssim \left\|\Pi_2(f^1, g)\right\|_{H_{\log}^M}+
    \left\|\Pi_2(f^d, g)\right\|_{H_{\log}^M} .$$
    The remainder of the proof is straightforward.
    This finishes the proof of Theorem \ref{thm-Doob}.
\end{proof}

\begin{remark}
To show that the two Hardy spaces $H_{\log}^M$ and $H_{\log}$ do not coincide
in general, one can apply a similar argument to the one
used in the proof of \cite[Proposition 2.16]{We1994}.
\end{remark}

Simplifications occur under assumptions on the filtration.
Recall that a filtration $(\mathcal{F}_n)_{n\in\mathbb{Z}_{+}}$
is said to be \emph{regular} if there exists a positive
constant $C_{\rm reg}$ such that, for any $n\in\mathbb{N}$ and $A\in\mathcal{F}_n$,
there exists a set $B\in\mathcal{F}_{n-1}$ such that $A\subset B$ and
$$
	\mathbb{P}(B)\leq C_{\rm reg}\mathbb{P}(A).
$$
Equivalently  (see, for instance, \cite[p.\,265]{Lo1993}),
for any $\mathcal F_n$ measurable non-negative function $f$,
$$f\leq C_{\rm reg}\E_{n-1}(f).$$
We will provide  concrete examples of martingales in Examples \ref{Nrf}
and \ref{Nrdf} below, which include both regular and non-regular martingales.

If the filtration $(\mathcal{F}_n)_{n\in\mathbb{Z}_{+}}$ is regular,
we have $s(f)\lesssim S(f),$ so that
 there exists a constant $C\in[1,\infty)$ such that, for any $f\in h_1$,
\begin{align}\label{hH}
	C^{-1}	\|f\|_{h_1}\leq \|f\|_{H_{1}}\leq C \|f\|_{h_{1}}
\end{align}
and, for any $f\in h_{\log}$,
\begin{align}\label{hHlog}
	C^{-1}	\|f\|_{h_{\log}}\leq \|f\|_{H_{\log}}\leq C \|f\|_{h_{\log}}.
\end{align}
See, for instance, \cite[Corollary 2.23]{We1994} and \cite[Theorem 2.5]{MNS2012}.

\begin{remark}\label{rem-bmo}
	Let $p,q\in[1,\infty)$.
If the relevant $(\mathcal{F}_n)_{n\in\mathbb{Z}_{+}}$ is regular,
	then all spaces ${\rm BMO}_p$ and ${\rm bmo}_q$ are equivalent;
	see \cite[Corollary 2.51]{We1994} for more details.
\end{remark}

Using Theorem \ref{thm-bilinear}, \eqref{hH}, \eqref{hHlog}, and Remark \ref{rem-bmo},
we obtain the following bilinear decomposition.
\begin{corollary}\label{coro-bmo}
If the filtration $(\mathcal{F}_n)_{n\in\mathbb{Z}_{+}}$ is regular,
then Theorem \ref{thm-bilinear} holds true when the spaces $H_1$,
${\rm BMO}$, and $H_{\rm log}$ therein are replaced, respectively, by
the spaces $h_1$,
${\rm bmo}$, and $h_{\rm log}$.
\end{corollary}

 As we can see, in the proof of Theorem \ref{thm-bilinear}, Lemma \ref{key-lem}
 plays an important role in the estimation of the  operator $\Pi_2$.
It is well known that  John--Nirenberg inequality  is  not valid in general
for ${\rm bmo}_p$ in place of ${\rm BMO}.$  A counter-example is given in
\cite[Example 2.17]{We1994}. This is why we could not replace in general
the space  ${\rm BMO }$ by ${\rm bmo_2}$ in Theorem \ref{thm-bilinear}.

\subsection{Density and terminal values}

In defining  the product of martingales, one was tempted to
define it as the product of terminal values. But, as we have seen in
the introduction, the product $f_\infty g_\infty,$ which is well defined a.s.,
is not integrable in general. Martingales in $H_{\log}$ have no terminal value,
either. But one can use the density to give a sense at formulas,
using the following well-known lemma (see, for instance, \cite[p.\,42]{Lo1993}).
\begin{lemma}\label{density}
	The space of $L^2$-martingales is dense in $H_1.$
    \end{lemma}

    Now, assuming that $f$ is an $L^2$-martingale and
    $g$ is a ${\rm BMO}$ martingale, the function $f_\infty g_\infty$ is
    in $L^{p}$ for any $p\in(0,2)$ and $(\E_n(f_\infty g_\infty))_{n\in\mathbb Z_+}$
    is an $L^1$-martingale. We can, as before, write
    $$f\cdot g= L(f, g)+ G(f,g).$$
   Since $L(f, g)$ and $\Pi_1(f,g)$ have a terminal value in $L^1,$
   it follows that the same is valid for  $\Pi_2(f,g).$ We deduce immediately from \eqref{MS}
   that $\Pi_2(f,g)$ is an $L^p$-martingale for any $p\in(0,2)$ and,
   in particular, an $L^1$-martingale.
    The previous equality can as well be written
    \begin{equation}\label{equ:Pi}
        f_\infty g_\infty=(f\cdot g)_\infty
        = \Pi_1(f,g)_\infty+\Pi_2(f, g)_\infty+L(f, g)_\infty.
    \end{equation}
 We write as well $\Pi_j(f, g)$ for any $j\in \{1,2\}$
 in place of their terminal value when it makes sense.
Theorem  \ref{thm-bilinear} leads then to the following one
which deals with functions instead of martingales. Here,
as was proposed in Remark \ref{ffinf},  we identify an $L^1$
function $f$ with the martingale $(\E_n(f))_{n\in\mathbb Z_{+}}$
and say that it is in $\rm BMO$ (resp.  $H_1$ or $ H_{\log}$)
if this martingale is in $\rm BMO$ (resp. $ H_1$ or $ H_{\log}$).
With these symbols, we deduce from Theorem  \ref{thm-bilinear}
the following statement.
    \begin{corollary} \label{th: bilinear-limit}
    	There exist two bilinear operators and a positive
    	constant $C $ such that
    	$\mathbf L:\ L^2\times \rm BMO\to L^{1}$ and
    	$\mathbf G:\  L^2\times \rm BMO\to L^{3/2}$
    	such that, for any $\varphi\in L^2$ and $\psi\in \rm BMO,$
    $$\varphi \psi=\mathbf L(\varphi, \psi)+\mathbf G(\varphi, \psi)$$
 with  $$\|\mathbf L(\varphi, \psi)\|_{L^1}
 \leq C \|\varphi\|_{H_1}\|\psi\|_{\rm BMO}
\ \mathrm{and} \ \|\mathbf G(\varphi, \psi)\|_{H_{\rm log}}
 \leq C \|\varphi\|_{H_1}\|\psi\|_{\rm BMO}.$$
    \end{corollary}

\subsection{Multipliers}

 By duality, Theorem  \ref{thm-bilinear} leads to a theorem on pointwise
 multipliers of $\rm BMO,$ which we define now.
\begin{definition}
	For $X$ being a normed space of $\mathcal{F}$-measurable functions,
	an $\mathcal{F}$-measurable function $g$ is
	called a {\it pointwise multiplier on $X$}
	if the pointwise product $fg$ belongs to $X$
	for any $f\in X$ and if there exists some positive constant $C,$ independent of $f,$
	such that $$\|fg\|_X\leq C\|f\|_{X}.$$ We denote by ${\rm PWM}(X)$
	the set of all pointwise multipliers on $X.$
	For any $g\in \mathrm{PWM}(X)$, its norm is defined by setting
	$$\|g\|_{\mathrm{PWM}(X)}
	:=\sup_{f\in X,\, \|f\|_X\neq 0} \frac{\|fg\|_X}{\|f\|_{X}}.$$
\end{definition}

Applying Theorem \ref{thm-bilinear}, we obtain the following corollary.
\begin{corollary}\label{cor-bi}
Bounded functions in $(H_{\log})^*$ are pointwise multipliers of $\rm BMO.$
In particular, for regular filtrations, bounded functions in $( {h}_{\log})^*$
are pointwise multipliers of $\rm BMO.$
\end{corollary}
\begin{proof}
Let us first notice that, since $H_1$ is contained in $H_{\log},$
its dual space identifies with a subspace of $\rm BMO,$
and the dual of $L^1+H_{\log}$ identifies with
$L^{\infty}\cap (H_{\log})^*.$ Moreover,
if $b\in L^{\infty}\cap (H_{\log})^*$ and $f\in L^2,$
the duality is given by the scalar product in $L^2,$ and
$$|\E(bf)|\leq \|b\| \, \|f\|_{L^1+H_{\log}}.$$
Let $g\in \rm BMO.$ For any $f\in L^2,$ the product $bgf$ is in $L^{p}$ for any $p\in(0,2)$,
and using Theorem \ref{thm-bilinear}, we have
$$ |\E(bgf)|\lesssim  \|f\|_{H_{1}}\|g\|_{\rm BMO}.$$
Since $L^2$ is dense in $H_1,$ this means that $bg$ identifies with
a $\rm BMO$ function, that is, $b$ is a pointwise multiplier of $\rm BMO$.
The last statement on regular filtrations is a consequence of
Lemma \ref{duality} and \eqref{hHlog}.
This finishes the proof of Corollary \ref{cor-bi}.
\end{proof}

For the case of regular martingales a direct proof can be deduced from
the work \cite{NS2014} of Nakai and Sadasue. They also give a converse
in the particular case when all $\sigma$-algebras $\mathcal F_n$ are
generated by atoms. Recall that, for any $n\in\mathbb{Z}_{+},$ a
set $B\in\mathcal{F}_n$ is called an
\textit{atom} of $\mathcal{F}_n$ if there exists no subset $A\subset B$ with
$A\in\mathcal{F}_n$ satisfying $\mathbb{P}(A)<\mathbb{P}(B)$.
The {\it martingale Campanato space} ${\rm bmo}_{\log}$ is defined
to be the set of all the martingales $f\in L^2$ such that
$$\|f\|_{{\rm bmo}_{\log}}:=\sup_{n\in\mathbb Z_{+}}\sup_{A\in \mathcal{F}_n}
\frac{1}{\phi(\mathbb{P}(A))}\left[\frac{1}{\mathbb{P}(A)}
\int_A\left|f-\mathbb{E}_n(f)\right|^2\,d\mathbb{P}\right]^{\frac12}<\infty,$$
where $\phi(r):=\frac{1}{r\Phi^{-1}(1/r)}$ for any $r\in(0,\infty)$
and $\Phi$ is a concave function which is equivalent to
the function $r\mapsto\frac{r}{\log(e+r)}$ for any $r\in(0,\infty)$.
It was proved in \cite[Theorem 2.10]{MNS2012} that
$(h_{\log})^{\ast}={\rm bmo}_{\log}$ with equivalent norms.

In \cite[Corollary 1.5]{NS2014}, Nakai and Sadasue  identified the
 pointwise multipliers of martingale Campanato spaces, which,
 in our case, gives the following statement.

\begin{lemma}\label{lem-sharpNS}
Assume that the filtration $(\mathcal{F}_n)_{n\in\mathbb{Z}_{+}}$ is regular and, moreover,
every $\sigma$-algebra $\mathcal{F}_n$  for any $n\in \mathbb{N}$ is
generated by a countable collection of atoms
and $\mathcal{F}_0=\{\Omega,\emptyset\}$.
Then
$${\rm PWM}\left({\rm bmo}_{1}\right)={\rm bmo}_{\log}\cap L^{\infty}$$
with equivalent norms.
\end{lemma}
This means that, in some sense  and in this particular case,
Theorem \ref{thm-bilinear} is the best possible: the dual statement is the best possible.

\section{Endpoint estimates of commutators in the martingale setting}\label{sect-commutator}
In this section, we apply the bilinear decomposition established  in the previous section to
investigate the  endpoint estimate of commutators in the martingale setting.
We also provide the proofs of both Theorems \ref{bdT} and \ref{cmT}.
We first define the class of operators for which we will study commutators.

\subsection{A class of operators}
This subsection is devoted to proving Theorem \ref{bdT}.
We first introduce the class $\mathcal{K}_q$ of sublinear operators.
Recall that we call the
 {\it martingale jump} a function $g\in L^1$ for which there exists an $n\in\mathbb{Z}_+$
 such that $g$ is $\mathcal{F}_n$ measurable and $\mathbb E_{n-1}(g)=0$.
For any $ q\in[1,\infty)$, the {\it space} $L^{q,\infty}$ is defined to be the set of
all the measurable functions $f$ on $\Omega$ such that
$$\|f\|_{L^{q,\infty}}:=\sup_{t>0}t
\left[\mathbb{P}\left(\{x\in\Omega:\ |f(x)|>t\}\right)\right]^{\frac1q}<\infty.$$

\begin{definition}\label{def-K}
	Let $ q\in[1,\infty)$.	Denote by $\mathcal{K}_q$ the set of all the
	continuous sublinear operators $T$ on $L^2$ satisfying
	\begin{enumerate}[{\rm (i)}]
		\item $T$ is bounded from $H_1$ to $L^q$;
		\item $T$ is bounded from  $L^1$ to $L^{q,\infty}$;
		\item  if $a$ is a simple $(s,\infty)$-atom with respect to some
		$n\in\mathbb{Z}_+$, then, for any $b\in \mathrm{BMO}$,
		\begin{equation}\label{AsT}
			\|(b-b_{n-1})T(a)\|_{L^q}\leq C \|b\|_{\mathrm{BMO}}
		\end{equation}
 and
		\begin{equation}\label{comT}
\left\|[T,b_{n-1}](a)\right\|_{L^q} \leq C \|b\|_{\mathrm{BMO}};
	\end{equation}	
		\item if $g$ is a martingale jump with respect to some $n\in\mathbb{Z}_+$,
		then, for any $b\in \mathrm{BMO}$,
		\begin{equation}\label{AsTd}
			\|(b-b_{n-1})T(g)\|_{L^q}\leq C\|g\|_{L^1} \|b\|_{\mathrm{BMO}}
		\end{equation}
 and
		\begin{equation}\label{comTd}
			\left\|[T,b_{n-1}](g)\right\|_{L^q} \leq C \|g\|_{L^1}\|b\|_{\mathrm{BMO}},
	\end{equation}	
		where $C$ is a positive constant independent of $a$, $g$, and $b$.
	\end{enumerate}
	Denote by $\mathcal{K}_H$ the set of all the  $T\in \mathcal{K}_1$
	such that $T(f)\in L^1$ if and only if $f\in H_1$.
\end{definition}
Recall that an operator $T$ is said to be {\it sublinear}
if, for any functions $f, g$ and any scalars $\alpha, \beta,$
one has
$$|T(\alpha f + \beta g)| \leq |\alpha|\,|T f| +|\beta|\,|T g|.$$
It follows in particular that, for any functions $f,g,$
$$
    \big||Tf|-|Tg|\big|\leq |T(f-g)|.
$$
\begin{remark}\label{rem-s-ast}
	According to \eqref{hH}, if the filtration $(\mathcal{F}_n)_{n\in \mathbb{Z}_+}$ is regular,
	then the martingale space $H_1$ also has the atomic decomposition.
	Thus, in regular case, we do not need \eqref{AsTd} and \eqref{comTd}
	in Definition \ref{def-K}. Moreover, in this case,
	we can use simple $\infty$-atoms instead of simple
	$(s,\infty)$-atoms; see \cite[Corollary 2.23]{We1994}
	and \cite[Theorem 2.5]{We1994}.
	Thus, in regular case, to show \eqref{AsT} and \eqref{comT},
	it suffices to prove that \eqref{AsT} and \eqref{comT} hold
	true for any simple $\infty$-atom $a$ with respect
	to some $n\in\mathbb{Z}_+$.

 If moreover $\sigma$-algebras $\mathcal F_n$ are atomic,
 we do not need the assumptions \eqref{comT} and \eqref{comTd}
 because, in this case, $b_{n-1}$ is a constant on the support of an atom.
 This explains why Ky does not need these assumptions
 in the classical case \cite{K2013}.
\end{remark}
\begin{remark}\label{rem-boundedL2}
	Comparing with the definition of Ky in \cite{K2013}, we assume that $T$ is an
	already bounded operator on $L^2.$ It is not a problem for applications, for
	which it is always satisfied. This assumption has been added to be able to give
	a meaning to commutators on a dense subset. Ky \cite{K2013}
	uses finite atomic decompositions, which have not been developed
	in the context of martingales. In the opposite direction,
	we allow $q$ to give other values, not just $q=1$ as in \cite{K2013},
	 and hence we can also treat of fractional integral operators.
\end{remark}

Next, we show  that both the  Doob maximal operator $M$ and the square function $S$ are in $\mathcal{K}_1$.	
In Section \ref{exampleK}, we   provide more examples of sublinear operators that are in $\mathcal{K}_q$.
\begin{example}\label{pro-TMS}
	Let $T$ be the   Doob maximal operator $M$ or the square function $S$. Then $T\in \mathcal{K}_1$.
\end{example}
\begin{proof}
Definition \ref{def-K}(i) with $q=1$ is a consequence of the definition
of $H_1$ for $S$ and of Lemma \ref{lem-Davis} for $M$.
Both $M$ and $S$ are of weak type $(1,1)$ and hence satisfy Definition \ref{def-K}(ii)
with $q=1$; see \cite[Theorems 2.1.1 and 2.1.2]{Lo1993}.
To prove that b they satisfy both (iii) and (iv) of Definition \ref{def-K},
we first show that \eqref{AsT} and \eqref{AsTd} hold true for $T=M.$
The other case can be proved by a similar argument.
	Let $a$  be a simple $(s,\infty)$-atom  with respect to some  $A\in \mathcal{F}_n$
	with $n\in\mathbb Z_{+}$.
	By Lemma \ref{TaC}(i), we find that $\mathrm{supp}\,(M(a))\subset A$.	
	From both the Jensen inequality for conditional expectations and
	the H\"{o}lder inequality, we deduce that, for any $n\in\mathbb Z_+$ ($b_{-1}:=0$ for convenience),
	\begin{align*}	
		\left\|\left(b-b_{n-1}\right) M(a)\right\|_{L^1}
		&=\int_{A} |b-b_{n-1}|\cdot M(a) \,d\mathbb{P}\\
		&=\int_{A}\mathbb{E}_n\left( |b-b_{n-1}| \cdot M(a) \right) \,d\mathbb{P}\\
		&\leq \int_{A}\left[\mathbb{E}_n\left( |b-b_{n-1}|^2\right)\right]^{1/2}
		\cdot\left[\mathbb{E}_n \left(M(a)^2\right)\right]^{1/2} \,d\mathbb{P}\\
		&\leq \|b\|_{\mathrm{BMO}} \,
		\left\|\left[\mathbb{E}_n \left(M(a)^2\right)\right]^{1/2}\right\|_{L^2} \|\mathbf{1}_A\|_{L^2}.
	\end{align*}
	Using Lemma \ref{TaC}(ii),
	we obtain, for any $n\in\mathbb Z_+$,
	\begin{align*}
		\left\|\left[\mathbb{E}_n \left(M(a)^2\right)\right]^{1/2}\right\|_{L^2}
		= \| M(a)\|_{L^2}
		\leq 2\left[\mathbb{P}(A)\right]^{-1/2}.
	\end{align*}
	Thus, we have
	$$\left\|\left(b-b_{n-1}\right) M(a)\right\|_{L^1}\leq 2\|b\|_{\mathrm{BMO}}.$$
	This shows \eqref{AsT} holds true for $M$.
	
	Now, we assume that $g$ is a martingale jump with respect to some $n\in\mathbb{Z}_+$.
	Then $g$ is $\mathcal{F}_n$-measurable and $\mathbb E_{n-1}(g)=0$
($\E_{-1}(g):=0$ for convenience), so, $M(g)=|g|.$ Thus,
	we find that
	\begin{align*}	\|(b-b_{n-1}) M(g)\|_{L^1}&=\int_{\Omega} |b-b_{n-1}|\cdot |g| \,d\mathbb{P}\\
		&=\int_{\Omega} |g|\cdot\mathbb{E}_n( |b-b_{n-1}|   ) \,d\mathbb{P}\\
		&\leq \|b\|_{\mathrm{BMO}} \int_{\Omega}|g|\, d\mathbb{P}
		\leq \|g\|_{L^1}\|b\|_{\mathrm{BMO}}.
	\end{align*}
Thus, \eqref{AsTd} holds true for $M$.

 Finally, let us prove that \eqref{comT} and \eqref{comTd} also hold
 true for $M$. It is easily seen that,
for any $n\in\mathbb N,$ $M(b_{n-1}f)=b_{n-1}M(f),$
 when the function $f$ is such that $d_k(f)=0$ for any $k\in\{0,\ldots,n-1\}.$
 The two operators, $M$ and multiplication  by $b_{n-1}$, commute
 for those functions, which allows to conclude for these two properties.
	We then conclude from the above argument  that
	Definition \ref{def-K} with $q=1$ holds
	true for the Doob maximal operator $M$. The proof for $S$ is similar.
	This finishes the proof of Examples \ref{pro-TMS}.
\end{proof}

\subsection{Proof of Theorem \ref{bdT} and Operator $U$}
\begin{proof}[Proof of Theorem \ref{bdT}]
Let us now define  the commutator $[T,b]$ of both the sublinear operator
$T\in \mathcal K_q$ and $b\in \mathrm{BMO}.$ It is well defined  on $L^2$
by setting, for any $f\in L^2$ and $x\in \Omega,$
$$[T,b](f)(x):=T(bf-b(x)f)(x).$$
Each separated term (in the linear case) does not make sense for a function in
$H_1.$ But the function $[T,b](f)$ is well defined for $L^2$ as a function in
$L^p$ for any $p\in[1, 2).$ It will be defined on $H_1$ by continuity from the
dense space $L^2.$ We first need to find a priori estimates.

So we first prove the present theorem for any $f\in L^2.$
Once we have proved the adequate a priori estimate,
it extends automatically to $H_1.$

We first consider the linear case and use the paraproduct decomposition \eqref{equ:Pi}, so that
$$[T,b](f)=T(\Pi_1( f, b))+T(\Pi_2(f,b))+T(L(f,b))-b T(f).$$
Using Proposition \ref{pro-1} and Definition \ref{def-K}(ii),
we already know that $f\mapsto T(\Pi_1( f, b))$ extends into a
bounded operator from $H_1$ to $L^q$ with $q\in[1,\infty)$. It remains to consider
the other term, which may be written, for a general sublinear
operator $T\in \mathcal K_q$ with $q\in[1,\infty)$, as
\begin{align}\label{def-U}
	U(f,b)(x)=	T(\Pi_2(f,b)-b(x)f)(x).
\end{align}
In fact, as before, this quantity is not well defined at this moment in all
generality, but makes sense when $f$ is in $L^2.$ We will develop a priori
estimates on the dense subset of $L^2$ functions, so that $U$ is defined by
continuity. The main result for $U$ is the following.
\begin{lemma}\label{Ufg}
	Let $q\in[1,\infty)$ and $U$ be the same as in \eqref{def-U} with $T\in \mathcal{K}_q$.
	Then $U$ extends into a bounded operator from $H_1\times \mathrm{BMO}$ into $L^q$.
\end{lemma}
If we take this lemma for granted,
the proof of Theorem \ref{bdT} follows at once for $T$ linear.
Whenever $T$ is only sublinear, we let
$$R(f,b):=|U(f,b)|+|T(\Pi_1(f,b))|.$$
Then it is easy to show that
$$|T(\Pi_3(f,b))|-R(f,b)\leq |[T,b](f)|\leq R(f,b)+|T(\Pi_3(f,b))|$$
and then conclude the desired conclusion in the same way as the linear case.
This finishes the proof of Theorem \ref{bdT}.
\end{proof}

We now show Lemma \ref{Ufg} whose proof needs a series of lemmas.
  We begin with the following result.
\begin{lemma}\label{g-gn}
	Assume that $a$ is a simple $(s,\infty)$-atom with respect to
	some $n\in \mathbb{Z}_+$ and assume $A\in \mathcal{F}_n$.
	If $b\in \mathrm{BMO}$, then
	$$\left\|\Pi_2\left(a,b-b_{n-1}\right)\right\|_{H_1}\leq 2\|b\|_{\mathrm{BMO}}.$$
\end{lemma}
\begin{proof}
By the assumption on $a$, we find that $d_ka=0$ for any
$ k\in\{0,\ldots,n\}$. Thus, we have
	\begin{align*}
		&\|\Pi_2(a,b-b_{n-1})\|_{H_1}\\
		&\quad= \left\| \sum_{k\geq n+1}(b_{k-1}-b_{n-1})d_ka\right\|_{H_1}
		=\left\| \left(\sum_{k\geq n+1}|b_{k-1}-b_{n-1}|^2|d_ka|^2\right)^{1/2}\right\|_{L^1}\\
		&\quad\leq  \left\| \left(\sum_{k\geq n+1}|b_{k-1}-b|^2|d_ka|^2\right)^{1/2}
		\right\|_{L^1} +\left\| \left(\sum_{k\geq n+1}|b-b_{n-1}|^2|d_ka|^2\right)^{1/2}\right\|_{L^1}\\
		&\quad=: \mathrm{I}_1+ \mathrm{I}_2.
	\end{align*}
	We   first estimate $\mathrm{I}_1$.
	Note that $\mathrm{supp}\,(d_ka)\subset A$ for each $k\geq n+1$. Thus,
	by the H\"older inequality, we obtain
	\begin{align*}
		\mathrm{I}_1 \leq \left\| \left(\sum_{k\geq n+1}|b_{k-1}-b|^2|d_ka|^2\right)^{1/2}
		\right\|_{L^2}  \|\mathbf{1}_A\|_{L^2} .
	\end{align*}
	On the other hand, by both the definition of $\mathrm{BMO}$
	and Lemma \ref{TaC}(iv), we have
	\begin{align*}
		&\left\| \left(\sum_{k\geq n+1}|b_{k-1}-b|^2|d_ka|^2\right)^{1/2}
		\right\|_{L^2}^2\\
		&\quad= \sum_{k\geq n+1} \int_{\Omega}|b-b_{k-1}|^2 |d_ka|^2 \,d\mathbb{P}\\
		&\quad=\sum_{k\geq n+1} \int_{\Omega}\mathbb{E}_{k}
		\left(|b-b_{k-1}|^2\right) |d_ka|^2 \,d\mathbb{P}\\
		&\quad\leq \|b\|_{\mathrm{BMO}}^2 \|S(a)\|_{L^2}^2
		\leq \|b\|_{\mathrm{BMO}}^2  \left[\mathbb{P}(A)\right]^{-1}.
	\end{align*}
	Therefore,
	$\mathrm{I}_1\leq \|b\|_{\mathrm{BMO}}.$
	
Next, we estimate $\mathrm{I}_2$. Since $\mathrm{supp}\,(S(a))\subset A$
(see Lemma \ref{TaC}) and $A\in\mathcal{F}_n$,
it follows from both the H\"older inequality
	and Lemma \ref{TaC}(iii) that
	\begin{align*}
		&\left\| \left(\sum_{k\geq n+1}|b-b_{n-1}|^2|d_ka|^2\right)^{1/2}\right\|_{L^1}\\
		&\quad\leq \left\|\, |b-b_{n-1}|S(a)\mathbf{1}_A\right\|_{L^1}
		\leq \left\|\, |b-b_{n-1}|\mathbf{1}_A\right\|_{L^2}\left\| S(a)\right\|_{L^2}\\
		&\quad\leq \left[\int_{\Omega}\mathbb{E}_n \left(|b-b_{n-1}|^2\right)
		\mathbf{1}_A\,d\mathbb{P}\right]^{\frac12}
		\left\| S(a)\right\|_{L^2}\leq \|b\|_{\mathrm{BMO}}.
	\end{align*}
	This establishes the  desired inequality and
	hence finishes the proof of Lemma \ref{g-gn}.
\end{proof}
We come back to the operator $U.$
\begin{lemma}\label{Uag}
	Let $q\in[1,\infty)$ and $T\in \mathcal{K}_q$. Then there exists a positive constant $C$ such that
	\begin{enumerate}[{\rm (i)}]
		\item for any simple $(s,\infty)$-atom and any $b\in \mathrm{BMO}$,
		$$\|U(a,b)\|_{L^q}\leq C \|b\|_{\mathrm{BMO}};$$
		\item for any martingale jump $g$ and any $b\in \mathrm{BMO}$,
		$$\|U(g,b)\|_{L^q}\leq C \|g\|_{L^1}\|b\|_{\mathrm{BMO}}.$$
	\end{enumerate}
\end{lemma}

\begin{proof}
	We first show (i).
	Without loss of generality, we may assume that $a$ is a simple $(s,\infty)$-atom
	with respect to some $n\in\mathbb{Z}_+$ and $A\in \mathcal{F}_n$.
	Observe that $\Pi_2(a,b_{n-1})=ab_{n-1}$.
	By this observation, we rewrite $U(a,b)$ as that, for any $x\in \Omega$,
	\begin{align}\label{new}
		U(a,b) (x)&=T\left(\Pi_2(a,b-b_{n-1})
		+[b_{n-1}-b_{n-1}(x)]a+[b_{n-1}(x)-b(x)]a\right)
		\end{align}
	Since $T\in \mathcal{K}_q$, it follows from Definition \ref{def-K}
	that $T$ is bounded from $H_1$ to $L^q$.
	Using this, \eqref{AsT}, \eqref{comT}, and Lemma \ref{g-gn}, we then obtain
	\begin{align*}
		\|U(a,b)\|_{L^q}&\leq \|T(\Pi_2(a,b-b_{n-1}))\|_{L^q}	
		+ 	\|\,[T,b_{n-1}](a)\,\|_{L^q} \\
		&\quad+	\|(b-b_{n-1})T(a)\|_{L^q}
		\\
		&\lesssim  \|b\|_{\mathrm{BMO}}+  \|\Pi_2(a,b-b_{n-1})\|_{H_1}
		\lesssim \|b\|_{\mathrm{BMO}}.
	\end{align*}
	This proves (i).
	
	For any martingale jump $g$ with respect to some $n\in\mathbb{Z}_+$,
	we have
	$$\Pi_2(g, b-b_{n-1})=\sum_{k\in \mathbb{Z}_+} d_kg \mathbb{E}_{k-1}(b-b_{n-1})
	=g \mathbb{E}_{n-1}(b-b_{n-1})=0$$
	and $\Pi_2(g,b_{n-1})=gb_{n-1}$. Using \eqref{new}, \eqref{comTd},
	and \eqref{AsTd}, we then obtain (ii).
	This finishes the proof of Lemma \ref{Uag}.
\end{proof}

We finally prove Lemma \ref{Ufg}.
\begin{proof}[Proof of Lemma \ref{Ufg}.]
	Take $(f,b)\in H_1\times \mathrm{BMO}$.	
	By Lemma \ref{lem-Davis}, we find that there exists a decomposition
	$f=f^1+f^d$ such that \eqref{davis-dec} holds true.
Moreover, by Remark \ref{sup-davis}, we can assume that both $f^1$ and $f^d$ are $L^2$-martingales.
From Lemma \ref{lem-atomic-decom}, we infer that
	there exist a sequence $(a^k)_{k\in\mathbb{Z}}$ of simple
	$(s,\infty)$-atoms and a sequence $(\mu_k)_{k\in\mathbb{Z}}$ of
	real numbers  such that
	$$f^1=\sum_{k\in\mathbb{Z}}\mu_ka^k\ \mbox{a.s.}$$
Moreover, using Remark \ref{sup-atomic}, we can assume that the sum is convergent in $L^2$ and
	$$U\left(f^1, b\right)=\lim_{n\to \infty}U\left(\sum_{k=-n }^n\mu_ka^k, b\right).$$
Similarly, we also have
 $$U\left(f^d, b\right)
 =\lim_{n\to \infty}U\left(\sum_{k=-n }^nd_k(f^d), b\right).$$
 To prove that the limit defines a bounded sublinear operator on $H_1,$
 it is sufficient to show its  uniform boundedness
 when both $f^1$ and $f^d$ are replaced by finite sums.
	In this case, combining the previous equalities, Lemma \ref{Uag},
	and \eqref{davis-dec}, we  obtain
	\begin{align*}
		\|U(f,b)\|_{L^q}&\leq \left\|U\left(f^1,b\right)\right\|_{L^q}
		+\left\|U\left(f^d,b\right)\right\|_{L^q}
		\\
		&\leq \sum_{k\in\mathbb Z} \left|\mu_k\right|
		\left\|U\left(a^k,b\right)\right\|_{L^q} +\sum_{n\in\mathbb Z_{+}}
		\left\|U\left(d_n(f^d),b\right)\right\|_{L^q}\\
		&\lesssim  \left\|f^1\right\|_{h_1} \|b\|_{\mathrm{BMO}}
		+ \sum_{n\in\mathbb Z_{+}}\left\|d_n(f^d)\right\|_{L^1} \|b\|_{\mathrm{BMO}} \\
		&\approx \left(\left\|f^1\right\|_{h_1}
		+\left\|f^d\right\|_{h_1^d}\right) \|b\|_{\mathrm{BMO}}
		\lesssim \|f\|_{H_1} \|b\|_{\mathrm{BMO}}.
	\end{align*}
	This finishes the proof of Lemma \ref{Ufg}.
\end{proof}

We now give two corollaries. The first one is another way to write
Theorem \ref{bdT}. The other is a direct consequence.
\begin{corollary}\label{th:equiv}
	Let $q\in[1,\infty)$, $T\in \mathcal{K}_q$, and $b\in \mathrm{BMO}$.
	Then $[T, b](f)$ is in $L^1$ if and only if $T(L(b, f))$ is in $L^1.$
	\end{corollary}
\begin{corollary}\label{weak}
	Let $q\in[1,\infty)$, $T\in \mathcal{K}_q$, and $b\in \mathrm{BMO}$.
	Then there exists a positive constant $C$ such that, for any $f\in H_1$,
	$$\|[T,b](f)\|_{L^{q,\infty}} \leq C \|f\|_{H_1}. $$
\end{corollary}
\begin{proof}
Just use Theorem \ref{bdT}, the fact that $L(f,b)$ is in $L^1$,
and the weak $L^q$ estimate for $T.$
This finishes the proof of Corollary \ref{weak}.
\end{proof}

Let us write the effect of Theorem \ref{bdT} on the maximal operator.
Recall that $\mathcal{K}_H$ is the set of all the  $T\in \mathcal{K}_1$
such that $T(f)\in L^1$ if and only if $f\in H_1$. According to Example
\ref{pro-TMS}, we find that the Doob maximal function $M\in \mathcal{K}_H$.
One can check that
$$[M,b](f)= \sup_{n\in \mathbb{Z}_+}\left|[\mathbb{E}_n,b](f)\right|.$$
By Theorem \ref{bdT}, we immediately obtain the following result.
\begin{corollary}\label{bdEn}
For any  $(f,b)\in H_1\times \mathrm{BMO}$,
\begin{align*}
\sup_{n\in \mathbb{Z}_+}|\mathbb{E}_n(L(f,b))|-R(f,b)
&\leq\sup_{n\in \mathbb{Z}_+}\left|[\mathbb{E}_n,b](f)\right|\\
&\leq \sup_{n\in \mathbb{Z}_+}|\mathbb{E}_n(L(f,b))|+R(f,b),
\end{align*}
	where $L$ is the same as in \eqref{decom} and
	$R:\ H_1\times \mathrm{BMO}\to L^1 $ is a bounded bilinear operator.
\end{corollary}

\subsection{Martingale Hardy space $H_1^b$ and the proof of Theorem \ref{cmT}}
In this subsection, we introduce the martingale Hardy space $H_1^b$ by borrowing
some ideas from Ky \cite{K2013} in harmonic analysis.
We also prove Theorem \ref{cmT} in this subsection.

\begin{definition}\label{def-h1b}
	Let $b\in \mathrm{BMO}$.
The {\it martingale Hardy space $H_1^b$} is defined to be
the set of all the martingales $f$ such that
	$$\|f\|_{H_1^b}:=\|f\|_{H_1}\|b\|_{\mathrm{BMO}}
	+\left\|\sup_{n\in\mathbb{\mathbb{Z}}_{+}}
	\left|\left[\mathbb{E}_n,b\right](f)\right|\right\|_{L^1}<\infty.$$
\end{definition}

We establish the following characterizations of the space $H_1^b$.
\begin{theorem}\label{4ec}
	Let $b\in \mathrm{BMO}$ be non-constant. Then the following assertions are equivalent:
	\begin{enumerate}[{\rm (i)}]
		\item $f\in H_1^b;$
		\item $L(f,b)\in H_1;$
		\item $[T, b](f)\in L^1$ with $T\in \mathcal{K}_H$.
	\end{enumerate}
	Furthermore, if   one of the above conclusions holds true, then
	\begin{align*}
	\|f\|_{H_1^b}
	&= \|f\|_{H_1}\|b\|_{\mathrm{BMO}}
	+\left\|\sup_{n\in \mathbb{Z}_+}|[\mathbb{E}_n,b](f)|\right\|_{L^1}\\
		&\approx  \|f\|_{H_1}\|b\|_{\mathrm{BMO}}+\|L(f,b)\|_{H_1}\\
	&\approx  \|f\|_{H_1}\|b\|_{\mathrm{BMO}}+\|[T,b](f)\|_{L^1},
	\end{align*}
where the positive equivalence constants are independent of both  $f$ and $b$.
\end{theorem}

\begin{proof}
	(i) $\Longleftrightarrow$ (ii).
	By Corollary \ref{bdEn}, we find that
	 $\sup_{n\in\mathbb{Z}_{+}}|[\mathbb{E}_n,b](f)|\in L^1$
	 if and only if
$$\sup_{n\in\mathbb{Z}_{+}}|\mathbb{E}_n(L(f,b))|\in L^1,$$
	 which, combined with Lemma \ref{davis-ine}, further implies that
	 $\sup_{n\in\mathbb{Z}_{+}}|[\mathbb{E}_n,b](f)|\in L^1$
	 if and only if $L(f,b)\in H_1$.
	Thus, we have
	\begin{align*}
	\|f\|_{H_1^b}&= \|f\|_{H_1}\|b\|_{\mathrm{BMO}}
	+\left\|\sup_{n\in \mathbb{Z}_+}|[\mathbb{E}_n,b](f)|\right\|_{L^1}	\\
	&\approx \|f\|_{H_1}\|b\|_{\mathrm{BMO}}+\|L(f,b)\|_{H_1},
	\end{align*}
which further implies the equivalence between items (i) and (ii).
	
(ii) $\Longleftrightarrow$ (iii). From Theorem \ref{bdT}, we deduce that
$[T,b](f)\in L^1$ if and only if $T(L(f,b))\in L^1$.
Since $T\in \mathcal{K}_H$, it follows that $T(f)\in L^1$
if and only if $f\in H_1$; see Definition \ref{def-K}.
Thus, $[T,b](f)\in L^1$ if and only if $L(f,b)\in H_1$. Moreover,
	\begin{align*}
	\|f\|_{H_1}\|b\|_{\mathrm{BMO}}+\|L(f,b)\|_{H_1}
	\approx\|f\|_{H_1}\|b\|_{\mathrm{BMO}}+\|[T,b](f)\|_{L^1},
	\end{align*}
which further implies the equivalence between items (ii) and (iii),
and hence completes the proof of Theorem \ref{4ec}.
\end{proof}

Now, we show Theorem \ref{cmT}.

\begin{proof}[Proof of Theorem \ref{cmT}]
	Let $f\in H_1^b$. Then, by Theorem \ref{4ec}, we find that $L(f,b)\in H_1$.	
	Since $T\in \mathcal{K}_q$, from Definition \ref{def-K},
	we infer that $T$ is bounded from $H_1$ to $L^q$.
	Now, using Theorem \ref{bdT}, we conclude that
	\begin{align*}
	\|[T,b](f)\|_{L^q}&\leq  		
	\|T(L(f,b))\|_{L^q}+\|R(f,b)\|_{L^q}\\
	&\lesssim\| L(f,b)\|_{H_1}+\|f\|_{H_1}\|b\|_{\mathrm{BMO}}
	\lesssim \|f\|_{H_1^b},
	\end{align*}
which completes the proof of Theorem \ref{cmT}.
\end{proof}

\begin{remark}\label{rem-commutator}
	Let $b\in \mathrm{BMO}$. We point out that $\mathcal{Y}:=H_1^b$ is the
	largest subspace of $H_1$ such that, for any $T\in \mathcal{K}_H$, the
	commutator $[T,b]$ is bounded from $\mathcal{Y}$ to $L^1$. Indeed,  if
	$\mathcal{Y}$ is a subspace of $H_1$ such that  $[T,b]$ is bounded from
	$\mathcal{Y}$ to $L^1$, then any element $f\in \mathcal{Y}$ justifies
	that $[T,b](f)\in L^1$. Hence, Theorem \ref{4ec} implies $f\in H_1^b$,
	which means $\mathcal{Y}\subset H_1^b$.
\end{remark}

\subsection{Examples of class $\mathcal{K}_q$}\label{exampleK}

As we already stated in Example \ref{pro-TMS}, both the Doob maximal function and the square function
belong to $\mathcal{K}_H \subset \mathcal{K}_1$ defined in Definition  \ref{def-K}.
In this subsection, we give
typical operators that are in $\mathcal{K}_q$.

\subsubsection{Martingale transforms}
 Martingale transforms were first introduced by Burkholder \cite{Bu1966}.
 Nowadays, martingale transforms have proven a very powerful tool not
only in probabilistic situation but also in harmonic analysis;
see, for instance, \cite{Ba2010} and its references. Recently, commutators of
martingale transforms for non-regular martingales were studied in \cite{Tr2013}.
In this subsection, we show that every martingale transform belongs to $\mathcal{K}_1$.
Consequently, we can apply both Corollary \ref{weak} and Theorem \ref{cmT} to
study the endpoint estimate  of  commutators of martingale transforms.

Let $\varepsilon:=(\varepsilon_k)_{k\in\mathbb{Z}_{+}}$ be an adapted measurable process
(that is, $\varepsilon_k$ is $\mathcal{F}_{k}$-measurable for each $k\in\mathbb{Z}_{+}$)
with $$\sup_{k\in\mathbb{Z}_{+}}\|\varepsilon_k\|_{L^{\infty}}\leq 1.$$
We let $\varepsilon_{-1}=0$ for convenience.
The {\it  martingale transform} of the martingale $f$ related to
$\varepsilon$ is the martingale $T_{\varepsilon}(f)$
defined by $d_k(T_\varepsilon f):= \varepsilon_{k-1}d_k f$ for any $k\in\mathbb Z_{+}$. Since
$S(T_\varepsilon f)\leq S(f),$ it follows immediately that $T_\varepsilon$
is bounded on $L^p$ for any $p\in(1,\infty)$
and also bounded on $H_1$. In all these cases it can be
identified with the mapping that is induced on terminal values: with now
$f$ a function in $L^p$ with $p\in(1,\infty)$ or in $H_1,$ the function
$T_\varepsilon f$ is defined by setting
$$T_{\varepsilon} (f):= \sum_{k\in\mathbb{Z}_{+}} \varepsilon_{k-1} d_k f.$$
We still speak of the martingale transform.

We also define the {\it maximal martingale transform} $M\circ T_\varepsilon$ by
$$M\circ T_\varepsilon f
:=\sup_{n\in\mathbb{Z}_{+}}\left|\sum_{k=0}^n
\varepsilon_{k-1}d_kf\right|.$$

Recall that the martingale transform shares the following properties
(see \cite{Bu1966, Lo1993}):
\begin{enumerate}[{\rm (i)}]
	\item for any $f\in L^2$, $\|T_{\varepsilon} (f)\|_{L^2}
	\leq \|M\circ T_\varepsilon(f)\|_{L^2} \leq C\|f\|_{L^2}$;
	\item for any $f\in H_1$, $\|T_{\varepsilon} (f)\|_{L^1}
	\leq \|M\circ T_\varepsilon(f)\|_{L^1}\leq C \|f\|_{H_1}$;
	\item for any $f\in L^1$, $\|T_{\varepsilon}(f)\|_{L^{1,\infty}}
	\leq  \|M\circ T_\varepsilon(f)\|_{L^{1,\infty}}
	\leq C\|f\|_{L^1}$,
\end{enumerate}
here $C$ is a positive constant independent of $f$.

\begin{proposition}\label{suppTa}
	Let $T_{\varepsilon}$ be the martingale transform  and  $M\circ T_{\varepsilon}$ be
	the maximal martingale transform as above. Then both $T_{\varepsilon}$ and
	$M\circ T_{\varepsilon}$ are all in $\mathcal{K}_1$.
\end{proposition}

\begin{proof}
Let us show that  \eqref{AsT} and \eqref{AsTd} hold true for both $T_{\varepsilon}$
and $M\circ T_{\varepsilon}$. It is not hard to check that the supports of both $T_{\varepsilon}(a)$
and $M\circ T_{\varepsilon}(a)$ are contained in $A$ whenever  $a$ is a simple
$(s,\infty)$-atom with respect to some $n\in\mathbb{Z}_{+}$ and $A\in \mathcal{F}_n$.
Besides, if $g$ is a martingale jump with respect to some $n\in\mathbb{Z}_{+}$, then we have
$$|T_{\varepsilon}(g)|\leq |g|\	\mbox{ and }\
M\circ T_{\varepsilon} (g)\leq |g|.$$
Note that both $T_{\varepsilon}$ and $M\circ T_{\varepsilon}$ are bounded on $L^2$.
Applying the same argument as that used in the proof of Example \ref{pro-TMS},
we obtain both \eqref{AsT} and \eqref{AsTd} hold true for both $T_{\varepsilon}$
and $M\circ T_{\varepsilon}$. Finally, the two last properties
\eqref{comT} and \eqref{comTd} for both $T_{\varepsilon}$ and $M\circ T_{\varepsilon}$ are
consequences of commutation properties: with the previous symbols,
$T_\varepsilon ( b_{n-1}a)=b_{n-1}T_\varepsilon (a) $
and $T_\varepsilon ( b_{n-1}g)=b_{n-1}T_\varepsilon (g);$
the same holds true for $M\circ  T_{\varepsilon}.$
This finishes the proof of Proposition \ref{suppTa}.
\end{proof}

The following result is a direct consequence of both Corollary \ref{weak} and Theorem \ref{cmT}.
\begin{corollary} 	
	Let $T_{\varepsilon}$ be the martingale transform as above.
	Let	$b\in \mathrm{BMO}$ be a non-constant function.  Then
	\begin{enumerate}[{\rm (i)}]
		\item the commutator $[T_{\varepsilon},b]$ is bounded from $H_1$ to $L^{1,\infty};$
		\item the commutator $[T_{\varepsilon},b]$ is bounded from $H_1^b$ to $L^{1}$.
	\end{enumerate}
\end{corollary}

Before leaving this subsection, we will prove an
analogue of the examples of functions in $H_1^b$
given in the classical case. We need a supplementary definition.
\begin{definition}
Let $b\in \mathrm{BMO}$ be a non-constant function.
The atom $a$ with respect to $n\in\mathbb Z_{+}$ is
called a {\it $(b, \infty)$-atom} if it satisfies the additional property
$$\E_n(ba)=0.$$
Moreover, we denote $\mathcal H_1^b$ the space of all the functions $f$ in $H_1$ such that
$$f=\sum_{j\in\mathbb Z} \mu_j a^j\ \mbox{a.s. and } \sum_{j\in\mathbb Z}\left|\mu_j\right|<\infty,$$
where $(a^j)_{j\in\mathbb Z}$ are $(b, \infty)$-atoms.
    \end{definition}
\begin{proposition}\label{perez}
	For	$b\in \mathrm{BMO}$ a non-constant function, the space $\mathcal H_1^b$ is contained in $H_1^b.$
    \end{proposition}
\begin{proof}
	We will only sketch the proof because it asks
	for variants of the previous ones. The first ingredient
	is the fact that we can as well consider martingale transforms
	with values in a Hilbert space (see, for instance, \cite{He1974}).
	Consider in particular the martingale transform $\Gamma, $ with
	values in $\ell^2,$ given by $\Gamma f=\sum_{n\in\mathbb N} (d_nf)e_n,$ where $e_n$
	is the canonical basis of $\ell^2.  $
Since $\|\Gamma f\|_{\ell^2}=S(f),$ it follows that
$H_1$-martingales are characterized among
$L^1$-martingales  by the fact that $\Gamma f$ is in $L^1(\Omega, \ell^2).$
We will take for granted that the previous theorems are valid for vector valued theorems,
so that it is sufficient to prove that $\Gamma(L(f, b))$ belongs to $ L^1(\Omega, \ell^2)$
whenever $f$ belongs to $\mathcal H_1^b$. We start by proving it for a $(b, \infty)$-atom.
\begin{lemma}\label{lem-h1b}
Let	$b\in \mathrm{BMO}$ be a non-constant function and $a$
be a $(b, \infty)$-atom. Then $\Gamma(L(a, b))$ is in
$L^1(\Omega, \ell^2).$ Moreover, its $L^1(\Omega, \ell^2)$
norm is bounded by a uniform constant.
    \end{lemma}
\begin{proof}
 Assume that the $(b, \infty)$-atom $a$ is related to some
 $n\in\mathbb Z_{+}$. We deduce from Corollary \ref{th:equiv}
 that it is sufficient to show that $[\Gamma,b](a)$
 is in $L^1(\Omega, \ell^2)$ with uniform norm. It is
 even sufficient to show the same for $[\Gamma,b-b_{n-1}](a)$
  because of Property \eqref{comT}, or for $\Gamma((b-b_{n-1})a)$
  because of Property \eqref{AsT}. But this is a consequence
  of the fact that $(b-b_{n-1})a$ is in $H_1$
  with a uniformly bounded norm. Indeed, since
  it is a $(b, \infty)$-atom,
  $$\E_n((b-b_{n-1})a)=\E_n(ba)-b_{n-1}\E_n( a)=0.$$ Moreover,
  by the H\"older inequality, we obtain
 $$\E \left(\left|(b-b_{n-1})a\right|^p\right)
 \leq \left[\E \left({\mathbf{1}}_A |b-b_{n-1}|\right)^{2p}
 \right]^{\frac 12}
 \left[\E \left(|a|^{2p}\right)
 \right]^{\frac 12}\lesssim \left[P(A)\right]^{1-p},$$
where we used the fact that
$$\E \left({\mathbf{1}}_A |b-b_{n-1}|^{2p}\right)
\leq \E \left({\mathbf{1}}_A\right)\left[\sup_{n\in\mathbb N}\E_n(|b-b_{n-1}|)\right]^{2p}.$$
So $M((b-b_{n-1})a)$ is supported in $A$ and
hence is in $L^1$ with uniformly bounded norm.
This implies that $(b-b_{n-1})a$ is in $H_1$
with uniformly bounded norm, which is what we wanted to prove.
This finishes the proof of Lemma \ref{lem-h1b}.
 \end{proof}
 Let us come back to the proof of Proposition \ref{perez}.
 Let $f=\sum_{j\in\mathbb Z} \mu_j a^j$ a.s.
 with $\sum_{j\in\mathbb Z}|\mu_j|<\infty,$
where $(a^j)_{j\in\mathbb Z}$ are $(b, \infty)$-atoms. The sequence
$(\sum_{j=1}^N \mu_j\Gamma(L( a^j, b)))_{N\in\mathbb N}$ is a Cauchy sequence in $H_1,$
which converges to its limit in $L^{1, \infty}(\Omega, \ell^2), $
that is, $\Gamma(L(f, b)).$ So $L(f, b)$ is in $H_1$ and hence $f$
is in $H_1^{b}.$ This finishes the proof of Proposition \ref{perez}.
\end{proof}

\subsubsection{ Fractional integrals in the martingale setting}
In the martingale setting, both the fractional integral operator and its commutator
were first investigated in \cite{CO1985}. We also refer the reader to Nakai et al.
\cite{ANS2020} for some recent developments on this topic. Nowadays, it is
well known that the martingale fractional integral operator can be viewed
as a discrete model of the Riesz potential in harmonic analysis. In this subsection,
we show that  the fractional integral operator $I_{\alpha}$ belongs to $\mathcal{K}_{\frac{1}{1-\alpha}}$
provided $\alpha\in(0,1)$. Based on this, we can investigate the
endpoint estimate of commutators of  the fractional integral operator $I_{\alpha}$.
Here, for convenience, we only consider $(\Omega,\mathcal F,\mathbb P;
	(\mathcal{F}_n)_{n\in \mathbb{Z}_+})$  given as below.
	
	\begin{example}\label{Nrf}
		Let the number sequence $(p_k)_{k\in \mathbb{Z}_{+}}\subset \mathbb{N}$
		be such that $p_k\geq2$ for each $k\in \mathbb{Z}_{+}$. 	
		For any $n\in\mathbb{Z}_{+}$, let $P_n:=\prod_{k=0}^n p_k$,
		\begin{equation}\label{FnP}
		\mathcal{F}_n:=\sigma\{[kP_n^{-1},(k+1)P_n^{-1}):\  k\in\{0,\ldots, P_n-1\}\},
		\end{equation}
		and
		$$\mathcal{F}:=\sigma\left(\bigcup_{n\in\mathbb{Z}_+}\mathcal{F}_n\right).$$
		Let us equip the measurable space $([0,1),\mathcal{F})$ with the Lebesgue measure $\nu$.
		Then it is not hard to show that the filtration $(\mathcal{F}_n)_{n\in \mathbb{Z}_+}$
		is not regular whenever  $\sup_{k\in \mathbb{Z}_+}p_k=\infty$.
	\end{example}
	
	For each $n\in \mathbb{Z}_+$, denote by $A(\mathcal{F}_n)$ the set of all atoms in $\mathcal{F}_n$.
	For any $\alpha\in(0,\infty)$, any $n\in\mathbb{Z}_{+}$, and any martingale $f\in L^1$,
	the {\it fractional integrals}, $I_{\alpha,n}$ and $I_{\alpha}$, of $f$ are defined, respectively, by setting
\begin{align}\label{fd}
I_{\alpha,n}(f):=\sum_{k=0}^n \beta_{k}^{\alpha} d_kf\ \mbox{ and }\
	I_{\alpha}(f):=\sum_{k\in \mathbb{Z}_+} \beta_{k}^{\alpha} d_kf,
\end{align}
	where, for any $ k\in \mathbb{Z}_+$,
	$$\beta_k: =\sum_{B\in A(\mathcal{F}_k)} \mathbb{P}(B)\mathbf{1}_B$$
	and  $\beta_{-1}:=\beta_0$.
	We refer to \cite{RW2016,SY2022} for more details on the fractional integral $I_{\alpha}.$
	It is clear that, for any $ k\in \mathbb{Z}_+$, $\beta_k=\frac1{P_k}$. Hence,
	for any $n\in\mathbb{Z}_{+}$,
	$$I_{\alpha,n}(f):=\sum_{k=0}^n P_k^{-\alpha} d_kf\ \mbox{ and }\
	I_{\alpha}(f)=\sum_{k\in \mathbb{Z}_+} P_k^{-\alpha} d_kf.$$

	Note that Example \ref{Nrf} is a special case of the filtration
	studied in \cite{RW2016,SY2022}. Thus, from \cite[Theorem 2.11]{RW2016} and
	\cite[Lemma 1.2 and Theorem 0.1]{SY2022},
we infer that, for any $\alpha\in(0,1)$,
there exists a positive constant $C_{(\alpha)}$, depending only on $\alpha$, such that
\begin{enumerate}[{\rm (i)}]
	\item for any $f\in H_1(0,1)$, $\|I_{\alpha}(f)\|_{L^{\frac{1}{1-\alpha}}(0,1)}
	\leq C_{(\alpha)}\|f\|_{H_1(0,1)}$;
	\item for any $f\in L^1(0,1)$,
	$\|I_{\alpha}(f)\|_{L^{\frac{1}{1-\alpha},\infty}(0,1)} \leq C_{(\alpha)}\|f\|_{L^1(0,1)}$;
	\item for any $f\in L^p(0,1)$,
	$\|I_{\alpha}(f)\|_{L^{q}(0,1)} \leq C_{(\alpha)}\|f\|_{L^p(0,1)}$,
	where $p\in (1,\infty)$ and $1/p-1/q=\alpha$.
\end{enumerate}

\begin{proposition}\label{pro-fi}
	Let $\alpha\in(0,1)$. Then the fractional integral operator $I_{\alpha}$
	belongs to $\mathcal{K}_{\frac{1}{1-\alpha}}$.
\end{proposition}
\begin{proof}
	According to the above argument, it remains to show that \eqref{AsT}-\eqref{comTd}
	hold true for $I_{\alpha}$ with $q:=\frac{1}{1-\alpha}$.
	We first assume that  $a$ is a simple $(s,\infty)$-atom  with respect
	to some $n\in \mathbb{Z}_+$ and $A\in \mathcal{F}_n$.
	Take $r\in (\alpha, 1)$ such that $r-\alpha\in(0,{1}/{2})$.
	Note that $\mathrm{supp}\,(I_{\alpha}(a))\subset A$
	because $I_{\alpha}$ is still a martingale transform.
	Then the H\"{o}lder inequality gives us
	\begin{align*}
	\left\|(b-b_{n-1})I_{\alpha}(a)\right\|_{L^{\frac{1}{1-\alpha}}(0,1)}
	&\leq 	 \left\|(b-b_{n-1})\mathbf{1}_A\right\|_{L^{\frac{1}{r-\alpha}}(0,1)} 	
	\left\|I_{\alpha}(a)\right\|_{L^{\frac{1}{1-r}}(0,1)}\\
	&=:\mathrm{I}_1 \times \mathrm{I}_2.
	\end{align*}
	Since $A\in \mathcal{F}_n$, it follows that
	\begin{align*}
	\mathrm{I}_1^{\frac{1}{r-\alpha}}
	=\int_A \mathbb{E}_{n}\left(|b-b_{n-1}|^{\frac{1}{r-\alpha}}\right)
	d\nu\lesssim \|g\|_{\mathrm{BMO}(0,1)}^{\frac{1}{r-\alpha}}  \nu(A).
	\end{align*}
	On the other hand, for $\mathrm{I}_2$, by the boundedness of
	$I_{\alpha}$ from $L^{\frac{1}{1-(r-\alpha)}}(0,1)$
	to $L^{\frac{1}{1-r}}(0,1)$, we obtain
	\begin{align*}
	\mathrm{I}_2 \lesssim \|a\|_{L^{\frac{1}{1-(r-\alpha)}}(0,1)}
	\lesssim \left[\nu(A)\right]^{-(r-\alpha)},
	\end{align*}
	where the second inequality follows from  Lemma \ref{TaC}(iv)
	because $1<\frac{1}{1-(r-\alpha)}<2$.
	Combining the estimates of both $\mathrm{I}_1$ and $\mathrm{I}_2$,
	we then conclude that
	$$\left\|(b-b_{n-1})I_{\alpha}(a)\right\|_{L^{\frac{1}{1-\alpha}}(0,1)}
	\lesssim \|g\|_{\mathrm{BMO}(0,1)}.$$
	This shows \eqref{AsT}.
	
	In the remainder of the present proof, assume that $a$ is a martingale
	jump with respect to some $n\in\mathbb{Z}_{+}$. Then
	$$\left|I_{\alpha}(a)\right|\leq P_n^{-\alpha} |a|.$$
	Since $a$ is measurable with respect to $\mathcal{F}_n$,
	it follows that
	\begin{align*}	
	&\left\|(b-b_{n-1}) I_{\alpha}(a)\right\|_{L^{\frac{1}{1-\alpha}}(0,1)}^{\frac{1}{1-\alpha}}\\
	&\quad\leq \int_{0}^1 \left(|b-b_{n-1}|\cdot P_n^{-\alpha}|a|
	\right)^{\frac{1}{1-\alpha}} d\nu\\
	&\quad=\int_{\Omega}\left(P_n^{-\alpha} |a|\right)^{\frac{1}{1-\alpha}}
	\cdot\mathbb{E}_n\left( \left|b-b_{n-1}\right|^{\frac{1}{1-\alpha}}  \right ) d\nu\\
	&\quad\leq \|g\|_{\mathrm{BMO}(0,1)}^{\frac{1}{1-\alpha}} \left\|P_n^{-\alpha}a
	\right\|_{L^{\frac{1}{1-\alpha}}(0,1)}^{\frac{1}{1-\alpha}}
	\overset{(R)}{\lesssim}\|b\|_{\mathrm{BMO}(0,1)}^{
		\frac{1}{1-\alpha}} \|a\|_{L^1(0,1)}^{\frac{1}{1-\alpha}},
	\end{align*}
	where the inequality $(R)$ is due to \cite[Lemma 2.3(i)]{RW2016}.
	Thus, \eqref{AsTd} holds true.
	Finally, $I_\alpha$ is a martingale transform and satisfies the same commutation properties
	 which lead to
	\eqref{comT} and \eqref{comTd}.
	Thus, $I_{\alpha}\in \mathcal{K}_{\frac{1}{1-\alpha}}$.
	This finishes the proof of Proposition \ref{pro-fi}.
\end{proof}

\begin{remark}\label{stolyarov}
The authors would like to thank Dmitriy Stolyarov who indicated to us 
that the estimates (i)--(iii) ahead of  Proposition \ref{pro-fi} 
are not valid when $\beta_k$ is replaced by $\beta_{k-1}$ in 
the definition of the fractional integral \eqref{fd} and who attracted  
our attention to his joint article \cite{SY2022}.
\end{remark}

Since  $I_{\alpha}\in \mathcal{K}_{\frac{1}{1-\alpha}}$,
the following conclusion directly follows from both Corollary \ref{weak} and Theorem \ref{cmT}.
\begin{corollary}
Let $\alpha\in(0,1)$ and $I_{\alpha}$ be the fractional integral operator
as above.  Let	$b\in \mathrm{BMO}(0,1)$ be non-constant.  Then
\begin{enumerate}[{\rm (i)}]
	\item the commutator $[I_{\alpha},b]$ is bounded
	from $H_1(0,1)$ to $L^{\frac{1}{1-\alpha},\infty}(0,1)$;
	\item the commutator $[I_{\alpha},b]$ is bounded
	from $H_1^b(0,1)$ to $L^{{\frac{1}{1-\alpha}}}(0,1)$.
\end{enumerate}
\end{corollary}

\section{Applications to harmonic analysis}\label{sect-application}

In this section, we aim to apply the bilinear decomposition of commutators
established in previous sections to some problems of harmonic analysis.

\subsection{Dyadic Hilbert transform beyond doubling measures}\label{subsect-hilbert}

Recall that the boundedness of both the dyadic Hilbert transform and its
adjoint associated with Borel measures were first characterized by L\'opez-S\'anchez et al. \cite{LLP2014}.
Motivated by this,  we devote this subsection to studying the commutator of
the dyadic Hilbert transform and its adjoint beyond doubling measures.

Here, we work with $([0,1), \mathcal{F}; (\mathcal{F}_n)_{n\in \mathbb{Z}_+})$,
which corresponds to  Example \ref{Nrf} with $p_k=2$ for each $k\in\mathbb{Z}_+$ therein.
However, we equip the measurable space $([0,1), \mathcal{F})$ with a probability measure
$\mu$ which is not necessarily the Lebesgue measure. In this subsection, the related Lebesgue space,
martingale Hardy space, and BMO space are denoted,
respectively, by $L^p(\mu)$, $H_1(\mu)$, and $\mathrm{BMO}(\mu)$.

In what follows, for each $n\in \mathbb{Z}_+$, denote by $A(\mathcal{F}_n)$
the set of all the dyadic intervals in $\mathcal{F}_n$.
Let $A(\mathcal{F})=\bigcup_{n\in\mathbb{Z}_+}A(\mathcal{F}_n)$.
Given  $I\in A(\mathcal{F})$, we write $I_-$ and $ I_+$, respectively,
for the left and the right dyadic children of $I$.
For any $I\in A(\mathcal{F})$, let
$$ m(I) :=\frac{\mu(I_-)\mu(I_+)}{\mu(I)}
\quad \mbox{and}\quad  h_I:= \sqrt{m(I)}
\left[\frac{\mathbf{1}_{I_-}}{\mu(I_-)}-\frac{\mathbf{1}_{I_+}}{\mu(I_+)}\right].$$
The following lemma is straightforward; see \cite[(2.3)]{LLP2014}.
\begin{lemma}\label{meas}
    For any $n\in\mathbb Z_{+}$ and any dyadic interval $I\in \mathcal{F}_n$,
    the function $h_I$ is $\mathcal{F}_{n+1}$-measurable
    and $\mathbb{E}_{n}(h_I)=0$. Moreover
\begin{equation}\label{hI1}
	\|h_I\|_{L^1(\mu)}=2\sqrt{m(I)}\ \mbox{and}\
	\|h_I\|_{L^{\infty}(\mu)}\approx \frac{1}{\sqrt{m(I)}}
\end{equation}
with the positive equivalence constants independent of $I$.
\end{lemma}

The probability measure $\mu$ is said to be {\it $m$-increasing} (resp. {\it $m$-decreasing})
if there exists a positive constant $C$ such that, for any $I\in A(\mathcal{F})$,
\begin{equation}\label{mi}
	m(I)\leq C m\left(\widehat{I}\right)\
	\left[\mbox{resp.} \ Cm(I)\geq m\left(\widehat{I}\right)\right].
\end{equation}
Here and thereafter, the symbol $\widehat{I}$ stands for the {\it dyadic parent} of $I$.

The next example is taken from \cite[Section 4.1]{LLP2014} and
proves that the dyadic filtration may be non-regular
with respect to $\mu.$
\begin{example}[Non-regular dyadic filtration]\label{Nrdf}
	Let $(I_k)_{k\in\mathbb{Z}_+}$ be a decreasing sequence of dyadic intervals,
	for instance,  $I_k=[0,2^{-k})$ for any $k\in\mathbb{Z}_+$.
	Let $(I_k^b)_{k\in\mathbb{N}}$ be its dyadic brother.
	Let $(\alpha_k)_{k\in \mathbb{N}}$ and $(\beta_k)_{k\in \mathbb{N}}$ be such that
	$\alpha_1:=1/2$, $\alpha_k:=1-2^{-k^2}$ for any $k\geq2$,  and $\beta_k:=1-\alpha_k$
	for any $k\in \mathbb{N}$. Define $\mu$ recursively by setting $\mu(I_0):=1$ and,
	for each $k\in \mathbb{N}$,
$$
		\mu(I_k):=\alpha_k  \mu\left(I_{k-1}\right)\
		\mbox{and}\  \mu\left(I_k^b\right):=\beta_k  \mu\left(I_{k-1}\right).
$$
	For any $k\in \mathbb{N}$ and any dyadic interval $J\subset I_k^b$, define
	$$\mu(J):= \frac{\nu(J)}{\nu(I_{k}^b)} \mu\left(I_k^b\right),$$
	where $\nu$ denotes the usual Lebesgue measure on $[0,1)$.
	It was shown  in \cite[p.\,72]{LLP2014} that $\mu$ is $m$-increasing.
	However, $\mu$ is non-dyadically doubling.
	Hence, the filtration $(\mathcal{F}_n)_{n\in \mathbb{Z}_+}$
	equipped with the measure $\mu$ is not regular.
\end{example}

Following \cite[p.\,50]{LLP2014}, the {\it dyadic Hilbert transform} is defined by setting,
for any measurable function $f$ on $[0,1)$ and any $x\in [0,1),$
\begin{equation}\label{expr-Hd}
	H_{\mathscr{D}}(f)(x):=\sum_{k\in \mathbb{N}}\sum_{I\in A(\mathcal{F}_k)}
	\delta(I) \langle f, h_{\widehat{I}}  \rangle  h_I(x),
\end{equation}
where $\delta(I):=1$ if $I:=( \widehat{I} )_-$,
$\delta(I):=-1$ if $I:=(\widehat{I})_+$, and $\langle f, h_{\widehat{I}} \rangle:=\E(fh_{\widehat{I}})$.
The main result of this subsection is as follows.
\begin{proposition}\label{DHT}
	Let $\mu$ be an $m$-increasing probability measure  on $[0,1)$.
	Then the dyadic Hilbert transform  $H_{\mathscr{D}}$  belongs to  $\mathcal{K}_1$.
\end{proposition}

To show Proposition \ref{DHT}, we need several lemmas.
The following result is due to \cite[(3.4)]{LLP2014} and \cite[Theorem 2.5]{LLP2014}.
\begin{lemma}\label{thm-HD}
	Let $\mu$ be an $m$-increasing probability measure  on $[0,1)$. Then
	\begin{enumerate}[{\rm (i)}]
		\item $H_{\mathscr{D}}$ is bounded on $L^2(\mu)$ and, moreover, for any $f\in L^2(\mu),$
		$\|H_{\mathscr{D}}(f)\|_{L^2(\mu)}\leq 2\|f\|_{L^2(\mu)};$
		\item   $H_{\mathscr{D}}$ is bounded from $L^1(\mu)$ to $L^{1,\infty}(\mu)$.
	\end{enumerate}
\end{lemma}

Similarly to the proofs of both (i) and (ii) of Lemma \ref{TaC}, using Lemma \ref{thm-HD}(i),
we can show the following conclusions; we omit the details.
\begin{lemma}\label{HDsa}
	Let $\mu$ be an $m$-increasing Borel measure  on $[0,1)$.
Assume that $a$ is a simple $(s,\infty)$-atom with respect to some  $n\in \mathbb{Z}_+$
and some dyadic interval $Q\in A(\mathcal{F}_n)$.
	Then
	\begin{enumerate}[{\rm (i)}]
		\item $\mathrm{supp}\,(H_{\mathscr{D}}(a))\subset Q$;
		\item $\|H_{\mathscr{D}}(a)\|_{L^p(\mu)}\leq 2\left[\mu(Q)\right]^{1/p-1}$ for any $ p\in[1,2]$.
	\end{enumerate}
\end{lemma}

\begin{lemma}\label{HDwad}
	Let $\mu$ be an $m$-increasing Borel measure on $[0,1)$.
	Then there exists a positive constant $C$ such that, for any $w\in L_d^1(\mu)$,
	$$\|H_{\mathscr{D}}(w)\|_{L^1(\mu)}\leq C\|w\|_{L_d^1(\mu)}.$$
\end{lemma}
\begin{proof}
Without loss of generality, we may assume that  $w$ is a martingale jump
with respect to some $n\in \mathbb{Z}_+$. Let us prove that
$\langle w, h_{\widehat{I}} \rangle:=\E(w h_{\widehat{I}})$
is $0$ except when $I$ belongs to  $A(\mathcal{F}_n)$. Indeed,
if $I\in A(\mathcal{F}_k)$ with some $k\in\{0,\ldots,n-1\},$ since
	$h_{\widehat{I}}$ is $\mathcal{F}_{k}$-measurable,
	it follows that
 $$\E_n(w h_{\widehat{I}} )=h_{\widehat{I}}\E_n w=0.$$ 	
 Next, if $I\in A(\mathcal{F}_k)$ with $k\in\{n+1,n+2,\ldots\},$ then
 $$\mathbb{E}_{k-1}(w h_{\widehat{I}})
 = w\mathbb{E}_{k-1}(h_{\widehat{I}})=0,$$
where we used  Lemma \ref{meas} for the last inequality.
Thus, $\E(w h_{\widehat{I}})=0$ when $I$ does not belong to  $A(\mathcal{F}_n)$.

Form \eqref{expr-Hd}, it follows that
\begin{equation}\label{H-d}
	H_{\mathscr{D}}(w) =\sum_{I\in A(\mathcal{F}_n)}
	\delta(I) \langle w, h_{\widehat{I}}  \rangle  h_I.
\end{equation}	
By the above equality, \eqref{hI1} and \eqref{mi}, we conclude that
\begin{align*}
	\left\|H_{\mathscr{D}}(w)\right\|_{L^1(\mu)}
	&\leq  \sum_{I\in A(\mathcal{F}_n)} \left|\langle w, h_{\widehat{I}}
	\rangle \right| \left\|h_{I}\right\|_{L^1(\mu)} \\
	&\leq \sum_{I\in A(\mathcal{F}_n)}\left\|h_{\widehat{I}}\right\|_{L^{\infty}(\mu)}
	\left\|h_{I}\right\|_{L^1(\mu)}\int_{\widehat{I}}|w|\,d\mu \\
	&\lesssim\sum_{I\in A(\mathcal{F}_n)}\frac{\sqrt{m(I)}}{\sqrt{m(\widehat{I})}}
	\int_{\widehat{I}}|w|\,d\mu\\
    &\lesssim \sum_{I\in A(\mathcal{F}_n)}
	\int_{\widehat{I}}|w|\,d\mu\lesssim \|w\|_{L^1},
\end{align*}
	This finishes the proof of Lemma \ref{HDwad}.
\end{proof}
\begin{remark}\label{H-dgen}
	It follows from \eqref{H-d} that, for a general martingale $f$,
$$
	H_{\mathscr{D}}(f) =\sum_{n\in\mathbb Z_{+}}
	\sum_{I\in A(\mathcal{F}_n)}  \delta(I) \langle d_n(f), h_{\widehat{I}}  \rangle  h_I.
$$
Moreover, if $a$ is an atom or a jump
with respect to $n$ and if $g$ is $A(\mathcal{F}_{n-1})$-measurable,
then $H_{\mathscr{D}}(ga)=  gH_{\mathscr{D}}(a)$.
\end{remark}

\begin{proposition}\label{HDh1L1}
	Let $\mu$ be an $m$-increasing Borel measure  on $[0,1)$.
	Then the dyadic Hilbert transform $H_{\mathscr{D}}$ is bounded from $H_1(\mu)$ to $L^1(\mu)$.
\end{proposition}

\begin{proof}
	Let $f\in H_1(\mu)$. Then the desired assertion follows from a combination of
	Lemmas \ref{davis-dec}, \ref{lem-atomic-decom},  \ref{HDsa}, and \ref{HDwad}.
	Indeed, by these lemmas, we obtain
	\begin{align*}
		\|H_{\mathscr{D}}(f)\|_{L^1(\mu)}
		&\leq 	\left\|H_{\mathscr{D}}\left(f^1\right)\right\|_{L^1(\mu)}
		+	 \left\|H_{\mathscr{D}}\left(f^d\right)\right\|_{L^1(\mu)}\\
		&\leq  \sum_{k\in \mathbb{Z}} \left|\mu_k\right|  	\left\|H_{\mathscr{D}}
		\left(a^k\right)\right\|_{L^1(\mu)} +\sum_{n\in \mathbb{Z}_{+}} 	
		\left\|H_{\mathscr{D}}\left(d_n\left(f^d\right)\right)\right\|_{L^1(\mu)} \\
		&\lesssim \left\|f^1\right\|_{h_1(\mu)}+ \left\|f^d\right\|_{h_1^d(\mu)}\lesssim \|f\|_{H_1(\mu)},
	\end{align*}
	where $(a^k)_{k\in \mathbb{Z}}$ is  a  sequence of simple $(s,\infty)$-atoms.
	This finishes the proof of Proposition \ref{HDh1L1}.
\end{proof}

Now, we are ready to show Proposition \ref{DHT}.

\begin{proof}[Proof of Proposition \ref{DHT}]
The second assertion of Remark \ref{H-dgen} implies \eqref{comT}
and \eqref{comTd}. So, with the help of both Lemma \ref{thm-HD}
and Proposition \ref{HDh1L1}, by Definition \ref{def-K} with $q=1$,
to complete the proof of present proposition,
it suffices to show that both \eqref{AsT}  with $q=1$ and \eqref{AsTd}  with $q=1$ hold true
for $H_{\mathscr{D}}$.  Using Lemma \ref{HDsa} and repeating the argument used
in the proof of Example \ref{pro-TMS}, we find that  \eqref{AsT} with $q=1$
holds true for $H_{\mathscr{D}}$. Now, assume that $b\in \mathrm{BMO}(\mu)$ and
$w\in L_d^1(\mu)$ is a martingale jump with respect to some $n\in \mathbb{Z}_+$.
	Then
	\begin{align*}	
		&\|(b-b_{n-1})H_{\mathscr{D}}(w)\|_{L^1(\mu)} \\
		&\quad\leq 	\|(b-b_{n})H_{\mathscr{D}}(w)\|_{L^1(\mu)}
		+	 \|(b_n-b_{n-1})H_{\mathscr{D}}(w)\|_{L^1(\mu)}\\
		&\quad=:\mathrm{I}_1+ \mathrm{I}_2.
	\end{align*}
	By  Lemma \ref{HDwad}, we have
$$\mathrm{I}_2\leq \|b_n-b_{n-1}\|_{L^{\infty}(\mu)} \|H_{\mathscr{D}}(w)\|_{L^1(\mu)}
\lesssim \|w\|_{L^1}\|b\|_{\mathrm{bmo}^d(\mu)}
\lesssim\|w\|_{L^1} \|b\|_{\mathrm{BMO}(\mu)}.$$
	Note that $H_{\mathscr{D}}(w)$, which is given by \eqref{H-d},
	is $\mathcal{F}_{n+1}$-measurable. Therefore,
	\begin{align*}
		\mathrm{I}_1&=    \int_{0}^1 \left[\mathbb{E}_{n+1}\left(|b-b_{n}|\right)\right]
		\times \left|H_{\mathscr{D}}(w)\right|\,d\mu\\
		&\leq \|b\|_{\mathrm{BMO}(\mu)} \left\|H_{\mathscr{D}}(w)\right\|_{L^1(\mu)}
		\lesssim \|w\|_{L^1}\|b\|_{\mathrm{BMO}(\mu)}.
	\end{align*}
Finally, we conclude from the estimates of both $\mathrm{I}_1$ and $\mathrm{I}_2$ that
	$$\left\|(b-b_{n-1})H_{\mathscr{D}}(w)\right\|_{L^1(\mu)}
	\lesssim \|w\|_{L^1}\|b\|_{\mathrm{BMO}(\mu)}.$$
Therefore, \eqref{AsTd} with $q=1$ also holds true for $H_{\mathscr{D}}$.
This finishes the proof of Proposition \ref{DHT}.
\end{proof}

Applying Proposition \ref{DHT}, the following corollary can be deduced
directly from both Corollary \ref{weak} and Theorem \ref{cmT}.
\begin{corollary}
	Let $\mu$ be an $m$-increasing Borel measure  on $[0,1)$
	and let $b\in \mathrm{BMO}(\mu)$ be non-constant. Then
	\begin{enumerate}[{\rm (i)}]
		\item the commutator $[H_{\mathscr{D}},b]$ is
		bounded from $H_1(\mu)$ to $L^{1,\infty}(\mu)$;
		\item the commutator $[H_{\mathscr{D}},b]$ is
		bounded from $H_1^b(\mu)$ to $L^{1}(\mu)$.
	\end{enumerate}
\end{corollary}

\begin{remark}
Denote by $H_{\mathscr{D}}^*$ the adjoint operator of
the dyadic Hilbert transform $H_{\mathscr{D}}$. 	
Let $\mu$ be an $m$-decreasing Borel measure  on $[0,1)$
and let $b\in \mathrm{BMO}(\mu)$ be non-constant. Similarly to the above corollary,
we can show that the commutator $[H_{\mathscr{D}}^*,b]$ is bounded from $H_1(\mu)$
to $L^{1,\infty}(\mu)$ and also from $H_1^b(\mu)$ to $L^{1}(\mu)$; we omit the details.
Besides,  one may also concern the $L^p$-boundedness, with $p\in(1,\infty)$, of
both the commutators $[H_{\mathscr{D}},b]$ and $[H_{\mathscr{D}}^*,b]$.
To limit the length of this article, we will not push this question here.
\end{remark}

\subsection{Ces\`{a}ro means of Walsh--Fourier series}\label{CWF}

In this subsection, we apply both Theorems \ref{bdT} and \ref{cmT} to study the commutator
of  the maximal operator of  Ces\`{a}ro means
of Walsh--Fourier series of functions in $L^1[0,1)$.
Throughout this subsection,  we work with
$([0,1), \mathcal{F}, \nu; (\mathcal{F}_n)_{n\in \mathbb{Z}_+})$,
where $\nu$ is the Lebesgue measure and each $\mathcal{F}_n$ with $n\in \mathbb{Z}_+$
is the same as in \eqref{FnP}  with $p_k=2$ for any $k\in \mathbb{Z}_+$.
Then the filtration $(\mathcal{F}_n)_{n\in \mathbb{Z}_+})$ is regular.

The following basic symbols in the Walsh--Fourier
analysis are taken from \cite{SWS1990,We2002}.
Every point $t\in [0,1)$ can be written as follows:
$$t=\sum_{k\in \mathbb{Z}_+} t_k 2^{-k-1}\mbox{ with } t_k\in\{0,1\} \mbox{ for any }k\in \mathbb{Z}_+.$$
If there are two different forms for the same $t$,
we choose the one for which $\lim_{k\to\infty}t_k=0$.
The {\it Rademacher functions} $(r_n)_{n\in \mathbb{Z}_+}$ on $[0,1)$
are defined by setting, for any $n\in \mathbb{Z}_+$,
$$r_n(t):= \exp(\pi i t_k), \ \forall\,t\in [0,1).$$
The product system generated by the Rademacher functions is the
Walsh system $(w_n)_{n\in \mathbb{Z}_+}$:
for any $n\in \mathbb{Z}_+$ with
$n=\sum_{k\in \mathbb{Z}_+} n_k 2^k$ ($n_k\in\{0,1\}$ for any $k\in \mathbb{Z}_+$),
\begin{equation*}
	w_{n}:=\prod_{k=0}^{\infty}r_k^{n_k}.
\end{equation*}
The {\it dyadic addition} $\oplus$ on $[0,1)$ is defined in the
following way: for any $t,s\in [0,1)$ with
$$t=\sum_{k\in \mathbb{Z}_+} t_k 2^{-k-1}\mbox{ and } s
=\sum_{k\in \mathbb{Z}_+} s_k2^{-k-1},$$
where $t_k,s_k\in\{0,1\}$ for any $k\in \mathbb{Z}_+,$ let
$$t\oplus s:=\sum_{k\in \mathbb{Z}_+} |t_k-s_k| 2^{-k-1}.$$
For any $t\in [0,1)$ and $I\in \mathcal{F}$, let
$$I\oplus t:=\{x\oplus t: x\in I\}.$$
By Theorem 4 in \cite[p.\,13]{SWS1990}, we find that
\begin{align}\label{lt}
	\nu(I\oplus t)=\nu(I), \ \forall\,t\in [0,1),
\end{align}
and, moreover, if $I=[0,2^{-n})$ for some $n\in \mathbb{N}$ and
$t=\sum_{k\in \mathbb{Z}}t_k 2^{-k-1}\in [\ell 2^{-n}, (\ell+1)2^{-n})$
for some $\ell\in \mathbb{Z}_+$, then
\begin{equation}\label{oplus-set}
	I\oplus t = \begin{cases}
		[\ell 2^{-n}, (\ell+1)2^{-n}), & \bar{t}\notin I, \\
		[\ell 2^{-n}, (\ell+1)2^{-n}], & \bar{t}\in I,
	\end{cases}
\end{equation}
where $\bar{t}:=\sum_{k\in \mathbb{Z}_+} \bar{t}_k2^{-k-1}$ with
\[
\bar{t}_k=\begin{cases}
	0, &  k\in\{0,\ldots,n-1\},\\
	|1-t_k|, & k\in\{n,n+1,\ldots\}.
\end{cases}\]

\begin{remark}
	Denote by $\mathbf{Z}_{2}$ the discrete cyclic group of order $2$,
	namely the set $\{0,1\}$ with the discrete topology and modulo $2$ addition.
	The {\it dyadic group} $G$  is then defined as the product
	$$G:=\prod_{k=0}^{\infty}\mathbf{Z}_{2},$$
	equipped with the product topology. We refer the reader to \cite[Chapter 1.3]{SWS1990}
	for the details  that the dyadic group $G$ can be identified with the
	interval $[0,1)$. Particularly, the Walsh functions on $[0,1)$
	can be viewed as the characters of the dyadic group $G$.
\end{remark}
For any $f \in L^1[0,1)$, the {\it $n$-th  Walsh--Fourier coefficient}
of $f$ is defined by setting, for any $n\in\mathbb{Z}_{+}$,
$$
\widehat {f}(n) := \int_0^1 f(t) w_{n}(t)\,d\nu(t).
$$
The {\it Walsh--Dirichlet kernels} $(D_n)_{n\in\mathbb N}$
are defined by setting, for any $n\in\mathbb N,$
$
D_n := \sum_{k=0}^{n-1} w_k
$
which satisfies
\begin{equation}\label{e5}
	D_{2^n}(x) = \left\{
	\begin{array}{ll}
		2^n &\ \hbox{if $x \in [0,2^{-n})$,} \\
		0 &\ \hbox{if $x \in [2^{-n},1)$.}
	\end{array}
	\right.
\end{equation}
For any $n\in\mathbb{N}$, denote by $S_{n}(f)$ the
{\it $n$-th partial sum} of the Walsh--Fourier series
of  $f\in L^1[0,1)$, that is, for any $x\in [0,1),$
$$
	S_{n}(f)(x) := \sum_{k=0}^{n-1} \widehat {f}(k)w_{k}(x)=f\ast D_n(x)
	=\int_0^1 f(t) D_{n}(x\oplus t)\, d\nu(t).
$$
For any $n\in \mathbb{Z}_+$, $f\in L^1[0,1)$, and $Q\in \mathcal F_n$, let
$$f_Q:=\frac{1}{\nu(Q)}\int_Qf\,d\nu.$$
Then, by the definition of the condition expectation, we obtain, for any $n\in \mathbb{Z}_+$,
$$\mathbb E_n(f)=\sum_{Q\in A(\mathcal{F}_n)}f_Q\mathbf{1}_Q.$$
From \eqref{e5}, it is easy to deduce that,
for any $n\in\mathbb N$ and $x\in [0,1)$,
$$
	S_{2^n}(f)(x)=\int_0^1 f(t) D_{2^n}(x\oplus t) \,d\nu(t)=\mathbb{E}_n(f)(x).
$$

The {\it Ces\`{a}ro means} $(	\sigma_{n})_{n\in\mathbb{N}}$
of $f\in L^1[0,1)$ are defined by setting,
for any $n\in\mathbb N$ and $x\in [0,1),$
$$
	\sigma_{n}(f)(x) :=f\ast K_n(x)=\int_0^1 f(t) K_n(x \oplus t)  \,d\nu(t),
$$
where $K_n$ denotes the {\it Walsh--Fej\'er kernel}  defined by setting
$$
K_n := \frac{1}{n} \sum_{k=1}^n D_n.
$$

Since the filtration $(\mathcal{F}_n)_{n\in \mathbb{Z}_+})$ is regular,
it follows from \cite[p.\,189]{NS2017} that
there exists a positive constant $C$ satisfying, for any $b\in \mathrm{BMO}[0,1)$,
$$\frac1C\|b\|_{\mathrm{BMO}[0,1)}\le \sup_{n\in\mathbb Z_{+}}
\sup_{Q\in A(\mathcal{F}_n)}\frac{1}{\nu(Q)}\int_Q\left|b-b_{Q}\right|\,d\nu
\le C\|b\|_{\mathrm{BMO}[0,1)}.$$

We first prove the following technical lemma.
\begin{lemma}\label{lem-cesa}
Let the maximal operator $\sigma:=\sup_{n\in\mathbb N}\sigma_n.$
Then there exists a positive constant $C$ such that,
for any $b\in \mathrm{BMO}[0,1)$ and any simple $\infty$-atom $a$
with respect to some $n\in \mathbb{Z}_+$ and some $Q\in A(\mathcal{F}_n)$,
$$\left\|\left(b-b_{\widehat{Q}}\right)\sigma(a)\right\|_{L^1[0,1)}
\leq C\|b\|_{\mathrm{BMO}[0,1)},$$
where $\widehat{Q}$ denotes the dyadic parent of $Q$.
\end{lemma}
\begin{proof}
Noting that the Lebesgue measure $\nu$ is translation invariant with respect to
the dyadic addition (see \eqref{lt} and also \cite[p.\,238]{We1996-AM}).
According to \eqref{oplus-set}, we find that,
for any $Q\in A(\mathcal F_n)$ with
$Q=[\ell 2^{-n}, (\ell+1)2^{-n})$ for some $\ell\in \mathbb{Z}_+$,
$Q=[0,2^{-n})\oplus t$ for some fixed $t\in Q$.
So, without loss of generality, we may assume that
$Q=[0,2^{-n})$.
We first write
\begin{align*}
	\left\|\left(b-b_{\widehat{Q}}\right)\sigma(a)\right\|_{L^1[0,1)}
	&= \int_Q \left|b-b_{\widehat{Q}}\right|\sigma(a) \,d\nu
	+ \int_{[0,1)\setminus Q}	\cdots \\
	&=:\mathrm{I}_1+\mathrm{I}_2.
\end{align*}
Since $\sigma$ is bounded on $L^{\infty}[0,1)$ (this can be easily
deduced from Theorem 3.4 in \cite{We2002})
and $a$ is a simple $\infty$-atom, it follows that
\begin{align}\label{cei}
	\mathrm{I}_1
	&\leq \|\sigma(a)\|_{L^{\infty}[0,1)}
	\int_Q \left|b-b_{\widehat{Q}}\right| \,d\nu\\
	&\lesssim\|a\|_{L^{\infty}[0,1)} 2^{-n+1} \|b\|_{\mathrm{BMO}[0,1)}
	\lesssim\|b\|_{\mathrm{BMO}[0,1)}.\nonumber
\end{align}

For the term $\mathrm{I}_2$, it was proved in  \cite[p.\,238]{We1996-AM} that,
for any $x\in [0,1)\setminus Q$,
\begin{align*}
	\sigma(a)(x)&\lesssim \sum_{j=0}^{n-1}\sum_{i=j}^{n-1}\left\{2^{i-n}
	\mathbf{1}_{[2^{-n},2^{-i})}(x)+2^{i-n}
	\mathbf{1}_{[2^{-j-1},2^{-j-1}\oplus 2^{-i})}(x)\right\}\\
	&\quad+2^{n} \sum_{j=0}^{n-1}2^{j}\sum_{i=n}^{\infty}
	2^{-i}\mathbf{1}_{[2^{-j-1},2^{-j-1}\oplus 2^{-n})}(x).
\end{align*}
Thus, we have
\begin{align*}
	\mathrm{I}_2&=\sum_{k=1}^n \int_{2^{-k}}^{2^{-k+1}}
	\left|b-b_{\widehat{Q}}\right|\sigma(a) \,d\nu\\
	&\leq\sum_{k=1}^n \sum_{j=0}^{n-1}
	\sum_{i=j}^{k-1} 2^{i-n} \int_{2^{-k}}^{2^{-k+1}}
	\left|b-b_{\widehat{Q}}\right| \,d\nu \\
	&\quad+
	\sum_{k=1}^n \sum_{j=0}^{n-1}\sum_{i=j}^{n-1}2^{i-n}
	\int_{2^{-k}}^{2^{-k+1}} \left|b-b_{\widehat{Q}}\right|
	\mathbf{1}_{[2^{-j-1},2^{-j-1}\oplus 2^{-i})}\,d\nu\\
	&\quad+ \sum_{k=1}^n 2^{n} 2^{k-1}\sum_{i=n}^{\infty}2^{-i}
	 \int_{2^{-k}}^{2^{-k+1}}
	\left|b-b_{\widehat{Q}}\right| \mathbf{1}_{[2^{-k},2^{-k}
		\oplus 2^{-n})}\,d\nu\\
	&=:\mathsf{B}_1+\mathsf{B}_2+\mathsf{B}_3.
\end{align*}
We recall that, if $Q\in A(\mathcal{F}_n) $ and $Q'\in A(\mathcal{F}_k)$
with $k\in \{0,\ldots,n\}$ such that $Q\subset Q'$, then
\begin{align}\label{basic}
	|f_{Q'}-f_Q|\lesssim (n-k)\|f\|_{\mathrm{BMO}[0,1)}.
\end{align}
This elementary inequality will be frequently used in the sequel.

To estimate $\mathsf{B}_1$, let  $Q_{k}:=[0,2^{-k})$ for every $1\leq k\leq n$.
Using \eqref{basic}, we obtain
\begin{align*}
	\mathsf{B}_1&\leq 2^{-n}\sum_{k=1}^n k2^k \int_{2^{-k}}^{2^{-k+1}}
	\left|b-b_{Q_{k-1}}+b_{Q_{k-1}}-b_{\widehat{Q}}\right| \,d\nu\\
	&\lesssim2^{-n}\sum_{k=1}^n k2^k \int_{2^{-k}}^{2^{-k+1}}
	\left|b-b_{Q_{k-1}}\right|\,d\nu+ 2^{-n}
	\sum_{k=1}^n k2^k 2^{-k}\left|b_{Q_{k-1}}-b_{\widehat{Q}}\right|\\
	&\lesssim 2^{-n}\left[\sum_{k=1}^n k
	+\sum_{k=1}^n k (n-k)\right]\|b\|_{\mathrm{BMO}[0,1)}
	\lesssim\|b\|_{\mathrm{BMO}[0,1)}.
\end{align*}

Now we estimate $\mathsf{B}_2$.
For any $ i,j\in\{1,\ldots,n\}$, let
$I_{i,j}:=[2^{-j},2^{-j}\oplus 2^{-i})=[0,2^{-i})\oplus 2^{-j}$.
It is clear that $\nu(I_{i,j})=2^{-i}$ and
$I_{i,j}\in A(\mathcal{F}_i)$ for each choice of $i,j$;
see \eqref{lt}. According to \eqref{oplus-set}, we have
\[I_{i,j+1}=\begin{cases}
	[0,2^{-i}), & i=j,\\
	[2^{-j-1}, 2^{-i}+2^{-j-1}), & i\geq j+1.
\end{cases}\]
Hence, by this, we have
\begin{align*}
	\mathsf{B}_2&= \sum_{k=1}^n \sum_{j=0}^{n-1}\sum_{i=j+1}^{n-1}2^{i-n}
	\int_{2^{-k}}^{2^{-k+1}} \left|b-b_{\widehat{Q}}\right|
	\mathbf{1}_{[2^{-j-1},2^{-j-1}+ 2^{-i})}\,d\nu\\
	&\quad\quad+\sum_{k=1}^n \sum_{j=0}^{n-1}2^{j-n} \int_{2^{-k}}^{2^{-k+1}}
	\left|b-b_{\widehat{Q}}\right| \mathbf{1}_{[0, 2^{-j})}\,d\nu\\
	&=:\mathsf{B}_{2,1}+\mathsf{B}_{2,2}.
\end{align*}
Observe that,
for any $i\geq j+1$, $I_{i,k}\subset [2^{-k}, 2^{-k+1})$  and
$I_{i,j}\cap [2^{-k},2^{-k+1})=\emptyset$ whenever $j\neq k$.
 Using this and \eqref{basic}, we obtain
\begin{align*}
	\mathsf{B}_{2,1}&=\sum_{k=1}^n \sum_{i=k}^{n-1}2^{i-n} \int_{I_{i,k}}
	\left|b-b_{I_{i,k}}+b_{I_{i,k}}-b_{Q_{k-1}}
	+b_{Q_{k-1}}-b_{\widehat{Q}}\right| \,d\nu\\
	&\leq \sum_{k=1}^n \sum_{i=k}^{n-1}2^{i-n} \int_{I_{i,k}} \left|b-b_{I_{i,k}}\right|
	\,d\nu +\sum_{k=1}^n \sum_{i=k}^{n-1}2^{i-n}2^{-i}
	\left|b_{I_{i,k}}-b_{Q_{k-1}}\right|\\
	&\quad+\sum_{k=1}^n \sum_{i=k}^{n-1}2^{i-n}2^{-i}
	\left|b_{Q_{k-1}}-b_{\widehat{Q}}\right|\\
	&\lesssim 2^{-n}\left[\sum_{k=1}^n
	\sum_{i=k}^{n-1}+\sum_{k=1}^n \sum_{i=k}^{n-1}(i-k)
	+\sum_{k=1}^n \sum_{i=k}^{n-1}(n-k)\right]\|b\|_{\mathrm{BMO}[0,1)}\\
	&\lesssim\|b\|_{\mathrm{BMO}[0,1)}.
\end{align*}
Next, by \eqref{basic}, we find that
\begin{align*}
	\mathsf{B}_{2,2}&\leq \sum_{k=1}^n
	\sum_{j=0}^{k-1}2^{j+1-n} \int_{2^{-k}}^{2^{-k+1}}
	\left|b-b_{Q_{k-1}}+b_{Q_{k-1}}-b_{\widehat{Q}}\right| \,d\nu\\
	&\lesssim\sum_{k=1}^{n}\sum_{j=0}^{k-1}2^{j+1-n}
	\left[2^{-k+1} +2^{-k}(n-k)\right]\|b\|_{\mathrm{BMO}[0,1)}\\
	&\lesssim\|b\|_{\mathrm{BMO}[0,1)}.
\end{align*}

Finally, for $\mathsf{B}_3$, it follows from both
the fact $I_{n,k}\subset Q_{k-1}$ and \eqref{basic} that
\begin{align*}
	\mathsf{B}_3
	&\leq 	\sum_{k=1}^n 2^k  \int_{I_{n,k}}
	\left|b-b_{I_{n,k}}+b_{I_{n,k}}-b_{Q_{k-1}}
	+b_{Q_{k-1}}-b_{\widehat{Q}}\right| \,d\nu\\
	&\lesssim \sum_{k=1}^n 2^k \int_{I_{n,k}}
	\left|b-b_{I_{n,k}}\right| \,d\nu
	+ \sum_{k=1}^n 2^k 2^{-n}\left|b_{I_{n,k}}-b_{Q_{k-1}}\right| \\
	&\quad+\sum_{k=1}^n 2^k 2^{-n}
	\left|b_{Q_{k-1}}-b_{\widehat{Q}}\right|\\
	&\lesssim\left[\sum_{k=1}^n2^{k}2^{-n}
	+ \sum_{k=1}^n2^{k}2^{-n}(n-k)\right]\|b\|_{\mathrm{BMO}[0,1)}\\
&	\lesssim \|b\|_{\mathrm{BMO}[0,1)}.
\end{align*}
Thus, we conclude that $\mathrm{I}_2\lesssim \|b\|_{\mathrm{BMO}[0,1)},$
which, together with \eqref{cei}, further implies the desired assertion.
This finishes the proof of Lemma \ref{lem-cesa}.
\end{proof}

Applying the above lemma, we obtain the endpoint estimate of
the commutator $[\sigma,b]$ with $b\in \mathrm{BMO}[0,1)$.

\begin{proposition}\label{pro-cesa}
Theorems \ref{bdT} and \ref{cmT} when $T=\sigma$ and $q=1$ hold true.
\end{proposition}
\begin{proof}
The following two inequalities can be found in Corollary 2 of \cite[p. 265]{SWS1990}
and \cite[Corollary 2.3]{MS2020} (see also \cite{We1996-AM}):
\begin{enumerate}[{\rm(i)}]
	\item for any $f\in L^1[0,1)$,
	$\|\sigma(f)\|_{L^{1,\infty}[0,1)}\leq C \|f\|_{L^1[0,1)}$,
	where $C$ is a positive constant independent of $f$;
	\item for any $f\in H_1[0,1)$,
	$\frac1C\|f\|_{H_1[0,1)} \leq\|\sigma(f)\|_{L^1[0,1)}
	\leq C\|f\|_{H_1[0,1)}$, where $C$ is a positive constant independent of $f$.
\end{enumerate}
Observe that, in the proofs of both Theorems \ref{bdT} and \ref{cmT},
\eqref{AsT} and \eqref{comT} are only used to show the following inequality:
for any simple $\infty$-atom $a$ and for any $b\in \mathrm{BMO}[0,1)$,
\begin{align}\label{cesa-u}
	\left\|U(a,b)\right\|_{L^1[0,1)}\lesssim \|b\|_{\mathrm{BMO}[0,1)},
\end{align}
where, for any $x\in[0,1)$,
 $$U(a,b)(x):=\sigma\left(\Pi_2(a,b)-b(x)a\right)(x).$$
Note that the filtration $(\mathcal{F}_n)_{n\in \mathbb{Z}_+}$
in this section is regular.
Thus, to prove the present proposition,
it is sufficient to show \eqref{cesa-u}.

For any $b\in \mathrm{BMO}[0,1)$ and any simple $\infty$-atom $a$
respect to some $n\in \mathbb{Z}_+$ and some $Q\in A(\mathcal{F}_n)$,
we obtain, for any $x\in[0,1)$,
\begin{align}\label{uce}
\left|U(a,b)(x)\right|&=
\left|U\left(a,b-b_{\widehat{Q}}\right)(x)\right|\\
&\leq \left|\sigma\left(\Pi_2\left(a,b-b_{\widehat{Q}}\right)
\right)(x)\right|+\left|b(x)-b_{\widehat{Q}}
\right|\sigma(a)(x).\nonumber
\end{align}
Since $\mathrm{supp}\,(a)\subset Q$, it follows that
$$\Pi_2\left(a,b-b_{\widehat{Q}}\right)=\Pi_2\left(a,b-b_{n-1}\right),$$
which, together with both the boundedness of $\sigma$ from $H_1[0,1)$
to $L^1[0,1)$ and Lemma \ref{g-gn}, further implies that
$$\left\|\sigma\left(\Pi_2\left(a,b-b_{\widehat{Q}}\right)
\right)\right\|_{L^1[0,1)}
\lesssim \left\|\Pi_2\left(a,b-b_{\widehat{Q}}\right)
\right\|_{L^1[0,1)}\lesssim \|b\|_{\mathrm{BMO}[0,1)}.$$
Combining this, \eqref{uce}, and Lemma \ref{lem-cesa},
we conclude that
\eqref{cesa-u} holds true.
This finishes the proof of Proposition \ref{pro-cesa}.
\end{proof}

The following conclusions follow  directly from both Proposition \ref{pro-cesa}
and Remark \ref{rem-commutator}; we omit the details.
\begin{corollary}
Let $b\in \mathrm{BMO}[0,1)$ be non-constant.  Then
	\begin{enumerate}[{\rm (i)}]
		\item the commutator $[\sigma,b]$ is bounded from $H_1[0,1)$ to $L^{1,\infty}[0,1)$;
		\item $\mathcal{Y}:=H_1^b[0,1)$ is the largest subspace of $H_1[0,1)$ such that
		the commutator $[\sigma,b]$ is bounded from $\mathcal{Y}$ to $L^1[0,1)$.
	\end{enumerate}
\end{corollary}

\noindent\textbf{Acknowledgement}.
We thank Odysseas Bakas, Zhendong Xu, Yujia Zhai,
and Hao Zhang for personal communication on this subject and for  sending to us  a preliminary version of \cite{BXZZ}. We also thank Quanhua Xu for his interest in the subject and for having indicated to us that Odysseas Bakas, Zhendong Xu, Yujia Zhai,
and Hao Zhang worked on related problems, which led us to send them a first draft of the content of this article.

Thanks go also to Dmitriy Stolyarov who indicated to us that  
we cited wrongly some results on the fractional integral 
(see Remark \ref{stolyarov}).

\bigskip

\noindent Aline Bonami

\smallskip

\noindent  Institut Denis Poisson, UMR CNRS 7013, University of Orl\'eans,
45067 Orl\'eans cedex 2, France

\smallskip

\noindent {\it E-mail}: \texttt{aline.bonami@univ-orleans.fr} (A. Bonami)

\bigskip

\noindent  Yong Jiao, Guangheng Xie and Dejian Zhou

\smallskip

\noindent  School of Mathematics and Statistics, HNP-LAMA,
Central South University, Changsha 410083, The People's Republic of China

\smallskip

\noindent {\it E-mails}: \texttt{jiaoyong@csu.edu.cn} (Y. Jiao)

\noindent\phantom{{\it E-mails}:} \texttt{xieguangheng@csu.edu.cn} (G. Xie)

\noindent\phantom{{\it E-mails}:} \texttt{zhoudejian@csu.edu.cn} (D. Zhou)

\bigskip

\noindent Dachun Yang

\smallskip

\noindent  Laboratory of Mathematics and Complex Systems
(Ministry of Education of China),
School of Mathematical Sciences, Beijing Normal University,
Beijing 100875, The People's Republic of China

\smallskip

\noindent {\it E-mail}: \texttt{dcyang@bnu.edu.cn} (D. Yang)


\begin{thebibliography}{10}

\bibitem{ANS2020}
	R.~Arai, E.~Nakai and G.~Sadasue, \emph{Fractional integrals and their
		commutators on martingale {O}rlicz spaces}, J. Math. Anal. Appl. \textbf{487}
	(2020),  Paper No. 123991, 35 pp.
	
	\vspace{-0.3cm}

\bibitem{BPRS2020}
O.~Bakas, S.~Pott, S.~Rodr{\'\i}guez-L{\'o}pez and A.~Sola, \emph{Notes on
	${H}^{\log}$: structural properties, dyadic variants, and bilinear $
	{H}^1$-${BMO}$ mappings}, Ark. Mat. \textbf{60} (2022), 231--175.

\vspace{-0.3cm}

\bibitem{BXZZ}
O.~Bakas, Z.~Xu, Y.~Zhai and H.~Zhang,
\emph{Multiplication between elements in martingale Hardy spaces and their duals},
arXiv: 2301.08723.

\vspace{-0.3cm}	

	\bibitem{Ba2010}
	R.~Ba\~{n}uelos, \emph{The foundational inequalities of {D}. {L}. {B}urkholder and some
		of their ramifications}, Illinois J. Math. \textbf{54} (2010),
	789--868 (2012).

\vspace{-0.3cm}

	\bibitem{BCKLYY2019}
	A.~Bonami, J.~Cao, L.~D.~Ky, L.~Liu, D.~Yang and W.~Yuan, \emph{Multiplication
		between {H}ardy spaces and their dual spaces}, J. Math. Pures Appl. (9)
	\textbf{131} (2019), 130--170.
	
\vspace{-0.3cm}

\bibitem{blyy21}
A. Bonami, L. Liu, D. Yang and W. Yuan,
\emph{Pointwise multipliers of Zygmund classes on $\mathbb R^n$},
J. Geom. Anal. \textbf{31} (2021), 8879--8902.

\vspace{-0.3cm}

\bibitem{BGK2017}
	A.~Bonami, J.~Feuto, S.~Grellier and L.~D.~Ky, \emph{Atomic decomposition and
		weak factorization in generalized Hardy spaces of closed forms},
	Bull. Sci. Math. \textbf{141} (2017),  676--702.

\vspace{-0.3cm}
	
	\bibitem{BGK2012}
	A.~Bonami, S.~Grellier and L.~D.~Ky, \emph{Paraproducts and products of functions
		in {${\rm BMO}(\mathbb R^n)$} and {$\mathcal H_1(\mathbb R^n)$} through wavelets}, J.
	Math. Pures Appl. (9) \textbf{97} (2012), 230--241.

\vspace{-0.3cm}
	
	\bibitem{BIJZ2007}
	A.~Bonami, T.~Iwaniec, P.~Jones and M.~Zinsmeister, \emph{On the product of
		functions in {BMO} and {$H_1$}}, Ann. Inst. Fourier (Grenoble) \textbf{57}
	(2007), 1405--1439.
\vspace{-0.3cm}

\bibitem{bk14}
A.~Bonami and L.~D.~Ky,
Factorization of some Hardy-type spaces of holomorphic functions,
C. R. Math. Acad. Sci. Paris 352 (2014), 817--821.
	
\vspace{-0.3cm}

	\bibitem{Bu1966}
	D.~L.~Burkholder, \emph{Martingale transforms}, Ann. Math. Statist. \textbf{37}
	(1966), 1494--1504.
	
\vspace{-0.3cm}

\bibitem{cky18}
J. Cao, L. D. Ky and D. Yang,
\emph{Bilinear decompositions of products of local Hardy and
	Lipschitz or BMO spaces through wavelets},
Commun. Contemp. Math. \textbf{20} (2018), 1750025, 30 pp.

\vspace{-0.3cm}

	\bibitem{CL1992}
J.-A.~Chao and R.~Long, \emph{Martingale transforms with unbounded multipliers},
Proc. Amer. Math. Soc. \textbf{114} (1992),   831--838.

\vspace{-0.3cm}

	\bibitem{CO1985}
	J.-A. Chao and H.~Ombe, \emph{Commutators on dyadic martingales}, Proc. Japan
	Acad. Ser. A Math. Sci. \textbf{61} (1985), 35--38.
	
\vspace{-0.3cm}

	\bibitem{CP1996}
	J.-A. Chao and L.~Peng, \emph{Schatten classes and commutators on simple
martingales}, Colloq. Math. \textbf{71} (1996), 7--21.
	
\vspace{-0.3cm}


	\bibitem{Da1970}
	B.~Davis, \emph{On the integrability of the martingale square function}, Israel
	J. Math. \textbf{8} (1970), 187--190.

\vspace{-0.3cm}

\bibitem{fyl17}
X. Fu, D. Yang and L. Liang,
\emph{Products of functions in ${\mathop\mathrm{BMO}\,}(\mathcal{X})$
and $H^1_{\rm at}(\mathcal{X})$ via wavelets
over spaces of homogeneous type}, J. Fourier Anal. Appl. \textbf{23} (2017), 919--990.	

\vspace{-0.3cm}

	\bibitem{Ga1973}
	A.~M. Garsia, \emph{Martingale Inequalities: {S}eminar Notes on Recent
		Progress}, Mathematics Lecture Note Series, W. A. Benjamin, Inc., Reading,
	Mass.-London-Amsterdam, 1973.
	
	\vspace{-0.3cm}

\bibitem{He1974}
C. Herz, \emph{Bounded mean oscillation and regulated martingales},
Trans. Amer. Math. Soc. \textbf{193} (1974), 199--215.
	
\vspace{-0.3cm}

	\bibitem{J1981}
	S.~Janson, \emph{B{MO} and commutators of martingale transforms}, Ann. Inst.
	Fourier (Grenoble) \textbf{31} (1981), 265--270.
	
\vspace{-0.3cm}

	\bibitem{K2013}
	L.~D. Ky, \emph{Bilinear decompositions and commutators of singular integral
		operators}, Trans. Amer. Math. Soc. \textbf{365} (2013), 2931--2958.

\vspace{-0.3cm}
	
	\bibitem{K2015}
	L.~D. Ky, \emph{Endpoint estimates for commutators of singular integrals
		related to {S}chr\"{o}dinger operators}, Rev. Mat. Iberoam. \textbf{31}
	(2015), 1333--1373.
	
\vspace{-0.3cm}

	\bibitem{LYY2018}
	L.~Liu, D.~Yang and W.~Yuan, \emph{Bilinear decompositions for products of
		{H}ardy and {L}ipschitz spaces on spaces of homogeneous type}, Dissertationes
	Math. \textbf{533} (2018), 1--93.
	
\vspace{-0.3cm}

	\bibitem{Lo1993}
	R.~Long, \emph{Martingale Spaces and Inequalities}, Peking University Press,
	Beijing; Friedr. Vieweg \& Sohn, Braunschweig, 1993.
	

	
\vspace{-0.3cm}

	\bibitem{LLP2014}
	L.~D. L\'{o}pez-S\'{a}nchez, J.~M. Martell and J.~Parcet, \emph{Dyadic
		harmonic analysis beyond doubling measures}, Adv. Math. \textbf{267} (2014),
	44--93.	
	
	\vspace{-0.3cm}
	
\bibitem{MS2020}
N. Memi\'c and S. Sadikovi\'c, \emph{Maximal operators and characterization of Hardy spaces},
Anal. Math. \textbf{46} (2020), 119--131.
	
\vspace{-0.3cm}

	\bibitem{MNS2012}
	T.~Miyamoto, E.~Nakai and G.~Sadasue, \emph{Martingale {O}rlicz-{H}ardy
		spaces}, Math. Nachr. \textbf{285} (2012), 670--686.
	
\vspace{-0.3cm}

	\bibitem{NS2014}
	E.~Nakai and G.~Sadasue, \emph{Pointwise multipliers on martingale {C}ampanato
		spaces}, Studia Math. \textbf{220} (2014),   87--100.
	
\vspace{-0.3cm}

	\bibitem{NS2017}
E.~Nakai and G.~Sadasue, \emph{Some new properties concerning BLO martingales},
Tohoku Math. J. (2) \textbf{69} (2017), 183--194.

\vspace{-0.3cm}

	\bibitem{NS2019}
	E.~Nakai and G.~Sadasue, \emph{Commutators of fractional integrals on martingale {M}orrey
		spaces}, Math. Inequal. Appl. \textbf{22} (2019), 631--655.
	
\vspace{-0.3cm}

	\bibitem{NY1985}
	E.~Nakai and K.~Yabuta, \emph{Pointwise multipliers for functions of bounded
		mean oscillation}, J. Math. Soc. Japan \textbf{37} (1985), 207--218.
	
	\vspace{-0.3cm}

 \bibitem{Per}
C.~P\'erez, \emph{Endpoint estimates for commutators of singular integral operators},
J. Func. Anal. \textbf{128}  (1995), 163--185.

\vspace{-0.3cm}	

\bibitem{PSTW22}
L.-E. Persson, F. Schipp, G. Tephnadze and F. Weisz,
\emph{An analogy of the Carleson--Hunt theorem with respect to Vilenkin systems},
J. Fourier Anal. Appl. \textbf{28} (2022), Paper No. 48, 29 pp.

\vspace{-0.3cm}	

\bibitem{PTW22}
L.-E. Persson, G. Tephnadze and F. Weisz,
\emph{Martingale Hardy Spaces and Summability of One-Dimensional Vilenkin--Fourier Series},
Birkhuser/Springer, Cham, 2022.

\vspace{-0.3cm}	

	\bibitem{Pe2019}
M. C. Pereyra,
\emph{Dyadic harmonic analysis and weighted inequalities: the sparse revolution},
in: New Trends in Applied Harmonic Analysis, Vol. 2, Harmonic Analysis,
Geometric Measure Theory, and Applications, pp. 159--239,
Appl. Numer. Harmon. Anal., Birkh\"auser/Springer, Cham, 2019.
	
\vspace{-0.3cm}

	\bibitem{RW2016}
	N.~Randrianantoanina and L.~Wu, \emph{Noncommutative fractional integrals},
	Studia Math. \textbf{229} (2015), 113--139.
	
\vspace{-0.3cm}

	\bibitem{SWS1990}
	F.~Schipp, W.~R. Wade and P.~Simon, \emph{Walsh Series}, Adam Hilger, Ltd.,
	Bristol, 1990, An Introduction to Dyadic Harmonic Analysis, With the
	collaboration of J. P\'{a}l.
	
\vspace{-0.3cm}

	\bibitem{SY2022}
	D.~Stolyarov and D.~Yarcev,
	\emph{Fractional integration for irregular martingales},
	Tohoku Math. J. (2) \textbf{74} (2022), 253--261.
	
	\vspace{-0.3cm}

	\bibitem{Tr2013}
	S.~Treil, \emph{Commutators, paraproducts and {BMO} in non-homogeneous
		martingale settings}, Rev. Mat. Iberoam. \textbf{29} (2013),
	1325--1372.
	
\vspace{-0.3cm}

\bibitem{VT}
A. L. Volberg and V. A. Tolokonnikov, \emph{Hankel operators and problems
	of best approximation of unbounded functions},
Translated from Zapiski Nauchnykh Seminarov Leningradskogo Matematicheskogo
Instituta im. V.A. Steklova AN SSSR,  {141} (1985), 5--17.
	
\vspace{-0.3cm}

	\bibitem{W1990}
F.~Weisz, \emph{Martingale {H}ardy spaces for {$0<p\leq 1$}}, Probab. Theory
Related Fields \textbf{84} (1990), 361--376.

\vspace{-0.3cm}

	\bibitem{We1994}
	F.~Weisz, \emph{Martingale {H}ardy Spaces and Their Applications in {F}ourier
		Analysis}, Lecture Notes in Mathematics, Vol. 1568, Springer-Verlag, Berlin,
	1994.
	
\vspace{-0.3cm}

	\bibitem{We1996-AM}
	F.~Weisz, \emph{Ces\`aro summability of one- and two-dimensional
		{W}alsh-{F}ourier series}, Anal. Math. \textbf{22} (1996), 229--242.
	
\vspace{-0.3cm}

	\bibitem{We2002}
	F.~Weisz, \emph{Summability of Multi-dimensional {F}ourier Series and {H}ardy
		Spaces}, Mathematics and its Applications, Vol. 541, Kluwer Academic
	Publishers, Dordrecht, 2002.

\vspace{-0.3cm}
	
	\bibitem{yyz21}
	D. Yang, W. Yuan and Y. Zhang,
\emph{Bilinear decomposition and divergence-curl estimates on products related
	to local Hardy spaces and their dual spaces},
	J. Funct. Anal. \textbf{280} (2021), Paper No. 108796, 74 pp.
	
	\vspace{-0.3cm}
	
	\bibitem{zyy22}
	Y. Zhang, D. Yang and W. Yuan,
	\emph{Real-variable characterizations
	of local Orlicz-slice Hardy spaces with application to bilinear
	decompositions}, Commun. Contemp. Math. \textbf{24} (2022),
Paper No. 2150004, 35 pp.
\end{thebibliography}
\end{document}